\documentclass{amsart}

\usepackage{amsmath,amsfonts,amsthm,amssymb,graphicx,tikz,color}
\usepackage[all,2cell,ps]{xy}

\theoremstyle{plain}
\newtheorem{thm}{Theorem}[section]
\newtheorem{lemma}[thm]{Lemma}
\newtheorem{prop}[thm]{Proposition}
\newtheorem{cor}[thm]{Corollary}
\newtheorem{conjec}[thm]{Conjecture}
\newtheorem{qtn}[thm]{Question}

\theoremstyle{definition}

\newtheorem{ex}[thm]{Example}

\newtheorem{rem}[thm]{Remark}

\theoremstyle{remark}

\DeclareMathOperator{\Aut}{Aut} 
 
 \DeclareMathOperator{\PSL}{PSL}
 
\DeclareMathOperator{\Ort}{O} \DeclareMathOperator{\PO}{PO}
 \DeclareMathOperator{\SO}{SO}

 \DeclareMathOperator{\Ad}{Ad}

\DeclareMathOperator{\id}{Id}

\DeclareMathOperator{\Isom}{Isom}\DeclareMathOperator{\spn}{span}
\DeclareMathOperator{\OF}{OF}

\newcommand{\al}{\alpha}

\newcommand{\ga}{\gamma}
\newcommand{\Ga}{\Gamma}

\newcommand{\De}{\Delta}

\newcommand{\la}{\lambda}

\newcommand{\sig}{\sigma}

\newcommand{\Fr}[1]{\ensuremath{\mathfrak{#1}}}

\newcommand{\fh}{\Fr h}
\newcommand{\fu}{\Fr u}

\newcommand{\fw}{\Fr w}

\newcommand{\Hy}{\ensuremath{\mathbb{H}}}

\newcommand{\Q}{\ensuremath{\mathbb{Q}}}
\newcommand{\R}{\ensuremath{\mathbb{R}}}
\newcommand{\Z}{\ensuremath{\mathbb{Z}}}

\newcommand{\calG}{\ensuremath{\mathcal{G}}}
\newcommand{\calH}{\ensuremath{\mathcal{H}}}

\newcommand{\calL}{\ensuremath{\mathcal{L}}}

\newcommand{\calO}{\ensuremath{\mathcal{O}}}

\newcommand{\calV}{\ensuremath{\mathcal{V}}}

\newcommand{\calX}{\ensuremath{\mathcal{X}}}

\newcommand{\ssm}{\smallsetminus}

\usetikzlibrary{arrows}

\title[Totally Geodesic Submanifolds]{Finiteness of Maximal Geodesic Submanifolds in Hyperbolic Hybrids}
\author[D Fisher]{David Fisher}
\address{Department of Mathematics\\Indiana University\\Bloomington, IN 47405}
\email{fisherdm@indiana.edu, nimimill@iu.edu}

\author[J.-F. Lafont]{Jean-Fran\c cois Lafont}
\address{Department of Mathematics\\Ohio State University\\Columbus, OH 43210}
\email{jlafont@math.ohio-state.edu}

\author[N Miller]{Nicholas Miller}

\author[M Stover]{Matthew Stover}
\address{Department of Mathematics\\Temple University\\Philadelphia, PA 19122}
\email{mstover@temple.edu}

\begin{document}

\begin{abstract}
We show that large classes of non-arithmetic hyperbolic $n$-manifolds, including the hybrids introduced by Gromov and Piatetski-Shapiro and many of their generalizations, have only finitely many finite-volume immersed totally geodesic hypersurfaces. In higher codimension, we prove finiteness for geodesic submanifolds of dimension at least $2$ that are \emph{maximal}, i.e., not properly contained in a proper geodesic submanifold of the ambient $n$-manifold.  The proof is a mix of structure theory for arithmetic groups, dynamics, and geometry in negative curvature.
\end{abstract}

\maketitle

\section{Introduction}\label{sec:Intro}

\noindent To simplify the discussion, throughout this introduction a hyperbolic manifold will mean a connected, oriented, complete, finite-volume hyperbolic $n$-manifold, and a geodesic submanifold will be a complete, immersed, finite-volume, totally geodesic subspace (note that a totally geodesic subspace need not be orientable). The main motivation of this paper is the following question which, as far as we know, has been asked independently by Reid and in dimension $3$ by McMullen \cite[Qn.\ 7.6]{Curt}:

\begin{qtn}\label{qtn:BigQ}
Let $M$ be a non-arithmetic hyperbolic manifold of dimension at least $3$. Does $M$ have at most finitely many geodesic submanifolds of codimension one?
\end{qtn}

\noindent We will prove the answer is yes for a large class of non-arithmetic hyperbolic manifolds that includes the famous non-arithmetic manifolds constructed by Gromov and Piatetski-Shapiro \cite{GPS}.  We also prove more general results valid in arbitrary codimension $k < n-1$, answering a more general question of Reid. Making a precise statement requires some more care and notation. Before embarking on that, we discuss some background and motivation for the conjecture.

\medskip

Totally geodesic submanifolds of hyperbolic manifolds, when they exist, have proven fundamental in solving a number of important problems. Perhaps most famously, Gromov and Piatetski-Shapiro used cut-and-paste of arithmetic hyperbolic manifolds along codimension one geodesic submanifolds to build non-arithmetic hyperbolic manifolds in all dimensions \cite{GPS}.  More recently, variants on their construction first introduced in \cite{SevenSamuraiShort} were used by Gelander and Levit to prove that ``most" hyperbolic manifolds in dimension at least $4$ are non-arithmetic \cite{GeLev}.

In another famous application, Millson used geodesic submanifolds to construct hyperbolic manifolds in all dimensions with positive $1^{st}$ betti number \cite{Millson}. Millson's examples are arithmetic, and a key idea in the proof is that if an arithmetic hyperbolic manifold contains a geodesic submanifold, then it contains many of them. One can see this as an easy application of the fact that the commensurator of an arithmetic group is dense in the isometry group of the associated symmetric space \cite[Th.\ 2]{BorelCrelle}. In particular, density of the commensurator allows one to prove the following:

\medskip

\noindent
\textbf{Arithmetic geodesic submanifold dichotomy}: \emph{For any $1 \le k \le n-1$, an arithmetic hyperbolic $n$-manifold either contains no codimension $k$ geodesic submanifolds, or it contains infinitely many and they are everywhere dense.}

\medskip

This was perhaps first made precise in dimension $3$ by Maclachlan--Reid and Reid \cite{MRTG, ReidTG}, who exhibited the first hyperbolic $3$-manifolds with no totally geodesic surfaces. (For non-arithmetic examples, see \cite{Calegari} for fibered knots with no totally geodesic surfaces.)

\medskip

In the non-arithmetic setting, this commensurator argument is not available. A theorem of Margulis shows that an irreducible lattice $\Gamma$ in an adjoint semisimple Lie group $G$ is arithmetic if its commensurator is topologically dense in $G$. From this one deduces that if $\Gamma$ is non-arithmetic, then its commensurator is itself a lattice that contains $\Gamma$ with finite index \cite[p.\ 2]{MargulisBook}, and so Question \ref{qtn:BigQ} is reasonable. The simplest form of our main result is the following.

\begin{thm}\label{thm:SimpleMain}
For every $n \ge 3$, there exist infinitely many commensurability classes of finite-volume non-arithmetic hyperbolic $n$-manifolds for which the collection of all codimension one finite-volume totally geodesic submanifolds is finite but nonempty.
\end{thm}

As we will describe below, Theorem \ref{thm:SimpleMain} applies in particular to the class of manifolds constructed by Gromov and Piatetski-Shapiro, along with many variants of their construction. In dimension $n \ge 4$, we do not know an example of a non-arithmetic hyperbolic $n$-manifold containing no codimension one geodesic submanifolds. Interestingly, finiteness results of this kind were only known previously for compact immersed totally geodesic surfaces in \emph{infinite volume} hyperbolic $3$-manifolds with compact convex core, by McMullen--Mohammadi--Oh \cite[Thm.\ 1.4]{MMO}. See very recent work of Benoist--Oh for the geometrically finite case \cite[Thm.\ 1.5]{BenoistOh}.

\medskip

Even for finite-volume hyperbolic $3$-manifolds, Theorem \ref{thm:SimpleMain} is new and the first result of its kind on the question of McMullen and Reid. All previous arguments in that setting could only show finiteness of geodesic subsurfaces by showing there were none at all. In fact, we will show in \S \ref{subsec:Links} that explicit examples of hyperbolic links with a finite but nonempty set of totally geodesic surfaces are relatively easy to construct.

\begin{thm}\label{thm:LinkMain}
There are infinitely many distinct hyperbolic links $L$ in the $3$-sphere so that the collection of all finite-area totally geodesic surfaces in $S^3 \ssm L$ is finite but nonempty. The link complements can be chosen to be mutually incommensurable. One can choose $L$ such that all its totally geodesic surfaces are noncompact.
\end{thm}

In short, our examples are \emph{belted sums} of arithmetic links. See \S \ref{sec:Exs} for more on various examples and nonexamples, along with connections between our work and the Menasco--Reid conjecture \cite{MenascoReid}.

\medskip

As mentioned above, Theorem \ref{thm:SimpleMain} is a special case of a more general theorem on geodesic submanifolds of arbitrary codimension. To state this result we require some more notation. A \emph{building block} of a hyperbolic $n$-manifold $M$ is a connected $n$-dimensional submanifold $N$ with nonempty (but possibly disconnected) totally geodesic boundary such that $\pi_1(N) < \pi_1(M)$ is Zariski dense in $\Isom^+(\Hy^n)$. In the finite-volume noncompact case, we require all cusps of $N$ to be finite-volume.

We say a hyperbolic manifold $M$ is \emph{built from building blocks} when it is the union of finitely many building blocks with disjoint interiors, hence they intersect only along their boundaries. That is, $M$ is obtained by gluing finitely many building blocks together isometrically along their boundaries. We call two such blocks \emph{adjacent} if they meet in $M$ along a common boundary component. We will frequently refer to any connected component of this common boundary as a \emph{cutting hypersurface}.

We say that a building block $N \subset M$ is \emph{arithmetic} when there exists an arithmetic hyperbolic $n$-manifold $M^\prime$ such that $N$ is isometric to a building block for $M^\prime$. We will see that, up to commensurability, $M^\prime$ is uniquely determined by $N$. In particular, it makes sense to say that two arithmetic building blocks are \emph{similar} when they are building blocks for commensurable arithmetic manifolds and are \emph{dissimilar} otherwise. We will also call a geodesic submanifold {\emph{maximal} if it is not contained in a proper geodesic submanifold of smaller codimension. Our main result is:

\begin{thm}\label{thm:Main}
Let $M$ be a non-arithmetic hyperbolic $n$-manifold built from building blocks for which there are two adjacent building blocks $N_1$ and $N_2$ that are arithmetic and dissimilar. For each $1 \le k \le n-2$, the collection of all maximal codimension $k$ finite-volume totally geodesic submanifolds of $M$ is finite.
\end{thm}

Note that the word maximal in the theorem is necessary, since codimension $k$ finite-volume totally geodesic submanifolds may be arithmetic and thus could contain infinitely many finite volume totally geodesics submanifolds of all greater codimensions. This occurs quite frequently in explicit constructions, particularly those built by Gromov and Piatetski-Shapiro.

Also note that the manifolds constructed by Gromov and Piatetski-Shapiro are built from two dissimilar buildings blocks (see \S \ref{ssec:GPS}), so Theorem \ref{thm:Main} applies. The theorem also applies to the manifolds used to study invariant random subgroups in \cite{SevenSamuraiShort} and to those used by Raimbault \cite{Raimbault} and Gelander--Levit \cite{GeLev} in studying growth of the number of maximal lattices in $\SO(n,1)$. These lattices are all built from \emph{subarithmetic pieces} in the language of Gromov and Piatetski-Shapiro \cite[Qn.\ 0.4]{GPS}.  We state one immediate corollary of our result and the results of Gelander--Levit more carefully:

\begin{thm}
Fix $n>3$.  Then there exists a number $b_n>0$ such that the number of commensurability classes of hyperbolic manifolds satisfying the conclusions of Theorem \ref{thm:Main} and having a representative of volume less than $V$ is proportional to $V^{b_n V}$.
\end{thm}

\noindent As in Gelander--Levit, it is possible to see from this that there are ``more" manifolds satisfying the conclusions of Theorem \ref{thm:Main} than there are arithmetic manifolds. In fact, the set of such manifolds has positive logarithmic density if one counts by minimal volume in the commensurability class.

\medskip

We note here that our proofs do not apply to the class of hyperbolic manifolds built by inbreeding that were introduced by Agol \cite{Agol} and developed further by Belolipetsky and Thomson \cite{BelThom}. Our approach depends at a key point on the adjacent pieces being associated with different arithmetic groups, and in this family of examples there is only one arithmetic group involved. In other words, in our terminology the building blocks used to construct the examples in \cite{Agol} and \cite{BelThom} are similar, hence Theorem \ref{thm:Main} does not apply.

\medskip

\noindent
The proof of Theorem \ref{thm:Main} is a mix of structure theory for arithmetic groups, dynamics, and geometry in negative curvature.

\medskip

Our proof involves three key ingredients.  The first two are a pair of `rigidity' theorems for geodesic submanifolds of manifolds $M$ satisfying the properties of Theorem \ref{thm:Main}. Each shows that geodesic submanifolds of $M$ that meet both $N_1$ and $N_2$ must be of a very particular kind. The third ingredient consists of equidistribution results from homogeneous dynamics.

The first ingredient is `Closure Rigidity', which we prove in \S \ref{sec:Closure}. Given a geodesic submanifold $\Sigma$ that crosses both $N_1$ and $N_2$, we show that $\Sigma \cap N_i$ is actually the intersection with $N_i$ of a geodesic submanifold in the finite volume arithmetic manifold from which $N_i$ is cut.  This is not true for geodesics, and in fact a generic geodesic in the examples we consider is built from pieces that do not close in the corresponding arithmetic manifold.

The second ingredient and rigidity result is `Angle Rigidity', proven in \S \ref{sec:Angle}. This says that once we know the closure rigidity results from \S \ref{sec:Closure}, geodesic submanifolds must meet $N_1 \cap N_2$ orthogonally.

The last key ingredient in the proof is a generalization of Shah's work on orbit closures in hyperbolic manifolds \cite{Shah}, which we prove in \S \ref{section:nimish}. This allows us to conclude that either the collection of codimension $k$ geodesic submanifolds of $M$ satisfies the conclusions of Theorem \ref{thm:Main}, or they determine a dense subset of the oriented orthonormal frame bundle of $M$. The combination of closure and angle rigidity forbids the second conclusion, hence the theorem follows.

There is some overlap of ideas between our paper and a recent paper of Benoist and Oh \cite{BenoistOh}.  In fact, after we showed our paper to Oh, she pointed out that in dimension 3 one can prove our Theorem \ref{thm:SimpleMain} using their Proposition 12.1 in place of our Theorem \ref{thm:GPSFiniteness}.  The proof in their paper does not cover the examples needed for Theorem \ref{thm:LinkMain}.  See Remark \ref{rem:BenoistOh} for the relation between our work and theirs.

\medskip

\noindent We close with some discussion of the following famous question asked by Gromov and Piatetski-Shapiro.

\begin{qtn}[Qn.\ 0.4 \cite{GPS}]\label{qtn:GPS}
Call a discrete subgroup $\Ga_0 < \PO(n,1)$ \emph{subarithmetic} if $\Ga_0$ is Zariski dense and if there exists an arithmetic subgroup $\Ga_1 < \PO(n,1)$ such that $\Ga_0 \cap \Ga_1$ has finite index in $\Ga_0$. Does every lattice $\Ga$ in $\PO(n,1)$ (maybe for large $n$) contain a subarithmetic subgroup? Is $\Ga$ generated by (finitely many) such subgroups? If so, does $V = \Hy^n / \Ga$ admit a ``nice'' partition into ``subarithmetic pieces''.
\end{qtn}

Manifolds built from arithmetic building blocks certainly admit a nice partition into subarithmetic pieces. In a different direction, in the 1960s Vinberg introduced the notion of a \emph{quasi-arithmetic} lattice \cite{Vinberg3} (see \S \ref{subsec:Coxeter}). For example, many lattices acting on $\Hy^n$ generated by reflections are quasi-arithmetic. However, the typical variants of Gromov and Piatetski-Shapiro's construction of non-arithmetic lattices are not quasi-arithmetic, so Question \ref{qtn:GPS} is well-known to be more subtle than ``are all non-arithmetic lattices iterations of the construction in \cite{GPS}''.

A key step in the proof that the examples in \cite{GPS} are non-arithmetic is the assumption that one interbreeds noncommensurable arithmetic manifolds. More recently, Agol \cite{Agol} (for $n=4$) and Belolipetsky--Thomson \cite{BelThom} (for $n \ge 4$) constructed non-arithmetic but quasi-arithmetic lattices by inbreeding commensurable arithmetic lattices. In particular, the existence of quasi-arithmetic lattices that are non-arithmetic does not rule out the possibility that all non-arithmetic lattices for $n \ge 4$ are constructed by inbreeding arithmetic lattices (there are hyperbolic $3$-manifolds with no totally geodesic surfaces, which clearly are not constructed via interbreeding or inbreeding; see Example \ref{ex:NoSurfacesNA}).

With the methods of this paper, we can give examples of Coxeter
polyhedra whose associated non-arithmetic lattices are not commensurable with either the Gromov--Piatetski-Shapiro type lattices or to the Agol type lattices. See \S \ref{subsec:Coxeter}, where we will prove the following, which perhaps further clarifies the subtle nature of Question \ref{qtn:GPS}.

\begin{thm}\label{thm:MainCoxeter}
There exist non-arithmetic lattices in $\SO(5,1)$ that are not commensurable with a lattice constructed by the methods of Gromov--Piatetski-Shapiro or Agol.
\end{thm}

\subsubsection*{Acknowledgments} Fisher was partially supported by NSF Grants DMS-1308291 and DMS-1607041. Lafont was partially supported by the NSF Grant Number DMS 1510640. Stover was partially supported by the NSF Grant Number DMS 1361000 and Grant Number 523197 from the Simons Foundation/SFARI. The authors acknowledge support from U.S. National Science Foundation grants DMS 1107452, 1107263, 1107367 ``RNMS: GEometric structures And Representation varieties'' (the GEAR Network).

\section{Background on constructions of hyperbolic $n$-manifolds}\label{sec:Background}

This section gives the necessary background on the constructions of hyperbolic $n$-manifolds that are relevant for this paper. It also establishes much of the notation that we will use throughout.

\subsection{Arithmetic manifolds constructed from quadratic forms}

Here we recall the construction of arithmetic hyperbolic manifolds via quadratic forms over number fields. Our standing assumption throughout will be that $K$ is a totally real number field of degree $s$ over $\Q$ with distinct real embeddings $\sigma_1,\dots,\sigma_s : K \to \R$. To simplify our discussion, we assume that $K$ is a subfield of $\R$ under $\sigma_1=\id$.

Now, let $V$ be a $K$-vector space of dimension $n+1$ and $q$ a nondegenerate $K$-defined quadratic form on $V$. Such a pair $(V,q)$ is called a $K$-\emph{quadratic space}, and associated with $q$ we have the inner product
\begin{equation}\label{eqn:InnerProd}
\langle u,v\rangle_q=\frac{1}{2}\left(q(u+v)-q(u)-q(v)\right).
\end{equation}
We then have the $K$-algebraic group $\SO(q,K)$ consisting of those determinant one linear transformations of $V$ that preserve $q$. We use
\[
V(\R) = V \otimes_K \R \cong \R^{n+1},
\]
to denote the extension of $V$ to a real vector space under our chosen real embedding of $K$, and $q$ will denote the unique extension of $q$ to a quadratic form on $V(\R)$. For the nonidentity $\sigma_i$, $q^{\sigma_i}$ will denote the extension of $q$ to the real vector space $V \otimes_{\sig_i(K)} \R$.

We assume throughout this paper that $q$ has signature $(n,1)$ on $V(\R)$ and that $q^{\sig_i}$ is definite for all $i\neq 1$. Given this setup, there is a natural injection
\begin{align*}
\mathfrak{i}:\SO(q,K)&\hookrightarrow\prod_{i=1}^s\SO(q^{\sigma_i},\R), \\
&\cong\SO(n,1)\times\prod_{i=2}^s\SO(n+1),
\end{align*}
sending $g \in \SO(q, K)$ to
\[
\mathfrak{i}(g) = (g,\sigma_2(g),\dots,\sigma_s(g)) \in \SO(n,1)\times\prod_{i=2}^s\SO(n+1).
\]
In other words, we identify $g \in \SO(q,K)$ with its image in $\SO(n,1)$ under projection onto the first factor of the above product.

Now, let $\calO_K$ denote the ring of integers of $K$. The above identifications map $\SO(q, \calO_K)$ to an arithmetic lattice $\Ga_q$ in $\SO(n,1)$. All arithmetic lattices in $\SO(n,1)$ considered in this paper will be commensurable with some $\Ga_q$ as above. We note that in even dimensions this construction determines every commensurability class of arithmetic lattices in $\SO(n,1)$.

\medskip

Given a quadratic form $q$ as above, we form the hyperboloid model of hyperbolic space associated with $q$ by
\[
\Hy_{q}^n=\{(x_1,\dots,x_{n+1})\in\R^{n+1}\mid q(x_1,\dots,x_{n+1})=-1,~x_{n+1}>0\}.
\]
The inner product on $\R^{n+1}$ defined by \eqref{eqn:InnerProd} determines the usual hyperbolic metric on $\Hy_q^n$. From here forward, when there is no possibility of confusion we will drop the subscript and identify $\Hy_q^n$ with $\Hy^n$.

A finite-volume hyperbolic orbifold is given as a quotient of $\Hy^n$ by a lattice $\Gamma$ in $\Isom(\Hy^n) \cong \PO(n,1)$. The group $\Isom^+(\Hy^n)$ of orientation-preserving isometries is isomorphic to the connected component $\SO_0(n,1)$ of the identity in $\SO(n,1)$. If $\Gamma$ is torsion-free then $\Hy^n/\Gamma$ is a manifold, and otherwise it is a hyperbolic orbifold. In particular, the arithmetic lattice $\Ga_q$ described above determines an arithmetic hyperbolic orbifold $\Hy^n/\Ga$. Selberg's lemma then implies that there is a finite index, torsion-free subgroup $\Gamma'<\Gamma$ and hence there is always an arithmetic hyperbolic manifold $\Hy^n/\Gamma'$ that finitely covers $\Hy^n/\Gamma$.

\medskip

For the reader's convenience, we also state the commensurability classification for arithmetic hyperbolic manifolds arising from the above construction. See \cite[2.6]{GPS} and \cite{Meyer} for further details.

\begin{thm}\label{thm:CommClassify}
Let $\Ga, \Ga^\prime$ be two arithmetic lattices in $\SO(n,1)$ constructed from quadratic spaces $(V,q), (V^\prime, q^\prime)$ over the number fields $K, K^\prime$, respectively. Then $\Ga^\prime$ is commensurable with $\Ga$ if and only if $K \cong K^\prime$ and $q^\prime$ is similar to $q$, i.e., there exists $\la \in K^*$ such that $(V^\prime, q^\prime)$ is isometric to $(V, \la q)$.
\end{thm}

We now briefly recall the construction of finite-volume codimension $k$ immersed totally geodesic submanifolds of arithmetic hyperbolic manifolds arising from quadratic forms. Let $\Ga$ be an arithmetic lattice in $\Isom(\Hy^n)$ associated with the quadratic space $(V, q)$ over the number field $K$. Let $W \subset V$ be a ($K$-defined) codimension $k$ subspace such that the restriction of $q$ to $W$ has signature $(n-k, 1)$ on $W(\R)$. Then the intersection of $W(\R)$ with $\Hy^n \subset V(\R)$ defines a totally geodesic embedding $f : \Hy^{n - k} \hookrightarrow \Hy^n$. The induced embedding of $\SO(W, q|_W)$ into $\SO(V, q)$ yields an embedding
\[
G_W = \mathrm{S}(\mathrm{O}(k) \times \mathrm{O}(n-k,1)) \cap \SO_0(n, 1) \hookrightarrow \Isom(\Hy^n),
\]
such that $\Ga_W = G_W \cap \Ga$ is a lattice in $G_W$. In particular, $\Hy^{n-k} / \Ga_W$ maps in as an immersed finite-volume totally geodesic submanifold of $\Hy^n / \Ga$. Moreover, it is well-known that all geodesic submanifolds of $\Hy^n / \Ga$ arise from the above construction, (see \cite{Meyer}).

\begin{rem}
It is also known that arithmetic manifolds of the above kind are the only arithmetic manifolds that contain a codimension one finite volume totally geodesic submanifold. In particular, this means that the arithmetic manifolds described above are precisely those that can be used in the constructions of non-arithmetic manifolds that follow.
\end{rem}

\subsection{Constructions of non-arithmetic manifolds following Gromov and Piatetski-Shapiro}\label{ssec:GPS}

In this section, we recall the construction by Gromov and Piatetski-Shapiro of non-arithmetic hyperbolic manifolds \cite{GPS}, along with generalizations by Raimbault \cite{Raimbault} and Gelander--Levit \cite{GeLev}.

\medskip

First, we recall the notion of a \emph{building block} and an \emph{arithmetic building block} from the introduction. As previously, a hyperbolic manifold will always mean a connected and oriented quotient of hyperbolic space with finite-volume. If $M$ is a hyperbolic $n$-manifold, then a building block $N \subset M$ is an $n$-dimensional submanifold with nonempty totally geodesic boundary such that $\pi_1(N) < \pi_1(M)$ is Zariski dense in $\Isom(\Hy^n)$ under the holonomy of the hyperbolic structure on $M$ and such that all cusps of $N$ have full rank. We call $N$ arithmetic when $M$ is an arithmetic hyperbolic manifold. The following is evident from \cite[1.6.A]{GPS}.

\begin{prop}\label{prop:BlockComm}
Suppose that $M$ is an arithmetic hyperbolic manifold and $N \subset M$ is an arithmetic building block. If $N \subset M^\prime$ for $M^\prime$ another arithmetic hyperbolic manifold, then $M^\prime$ is commensurable with $M$.
\end{prop}

In particular, if $N_1, N_2$ are arithmetic building blocks with associated arithmetic hyperbolic manifolds $M_1$ and $M_2$, it makes sense to say that they are \emph{similar} if $M_1$ and $M_2$ are commensurable and \emph{dissimilar} otherwise. We now return to constructing non-arithmetic hyperbolic manifolds.

\medskip

The most basic Gromov--Piatetski-Shapiro construction is a hybrid manifold defined by interbreeding two incommensurable arithmetic hyperbolic manifolds. In our language, one builds a hyperbolic manifold consisting of two dissimilar arithmetic building blocks. Specifically, suppose that $M_1$ and $M_2$ are incommensurable arithmetic hyperbolic $n$-manifolds and that $\Sigma$ is a hyperbolic $(n-1)$-manifold that admits totally geodesic embeddings
\[
f_i : \Sigma \hookrightarrow M_i,
\]
$i = 1,2$. Suppose that $f_i(\Sigma)$ separates $M_i$, and choose one side $N_i$. This is then a hyperbolic manifold with connected totally geodesic boundary isometric to $\Sigma$, and $N_i$ is an arithmetic building block in the sense described above. Gluing $N_1$ to $N_2$ along their boundaries then defines a non-arithmetic manifold comprised of two incommensurable arithmetic building blocks.

\medskip

Further variants of this construction are given by Raimbault \cite{Raimbault} and Gelander--Levit \cite{GeLev}, who use a generalization of this idea to provide asymptotic counts on the number of hyperbolic manifolds in terms of volume. We describe the constructions so the reader can explicitly see that they satisfy the assumptions of Theorem \ref{thm:Main}.

\medskip

Raimbault's variant of the Gromov--Piatetski-Shapiro construction glues together arithmetic building blocks in a circular pattern. More specifically, start with a family of $r$ arithmetic building blocks $N_1,\dots, N_r$ that are pairwise dissimilar and such that $\partial N_j$ has two connected components, each being isometric to some fixed hyperbolic $(n-1)$-manifold $\Sigma$ (with opposite orientations). Let $\Sigma_j^+$ (resp.\ $\Sigma_j^-$) denote the positively (resp.\ negatively) oriented boundary component. One then glues $N_j$ to $N_{j+1}$ by identifying $\Sigma_j^+$ to $\Sigma_{j+1}^-$, $1 \le j \le r-1$, and $N_r$ to $N_1$ by $\Sigma_r^+$ to $\Sigma_1^-$ to obtain a non-arithmetic hyperbolic manifold built from $r$ building blocks.

\medskip

Gelander--Levit also construct large families of non-arithmetic manifolds, in which they start with any 4-regular and 2-colored graph $G$ on $k$ vertices such that every vertex has the same color except for one and such that each edge is given a label from the set $\{a,a^{-1},b,b^{-1}\}$.
They then form a graph of spaces where the vertex set $V$ corresponds to two fixed arithmetic building blocks $\{V_0,V_1\}$ each with totally geodesic boundary with four connected components, each isometric to a fixed hyperbolic $(n-1)$-manifold $\Sigma$.
Moreover, the edge set $E$ corresponds to gluing together two of four arithmetic building blocks from a set $\{A^-,A^+,B^-,B^+\}$ each with totally geodesic boundary with two connected components, each also isometric to the same $\Sigma$ as before (see \cite[Rem.\ 3.2]{GeLev} for details).
Provided the building blocks are dissimilar, this constructs a non-arithmetic hyperbolic manifold.

\section{Closure rigidity}\label{sec:Closure}

The purpose of this section is to show that any totally geodesic submanifold with boundary lying on a cutting hypersurface can be extended to a finite-volume totally geodesic submanifold of the arithmetic manifold associated with each of the adjacent building blocks.

\medskip

Proving this will require the following lemma, whose statement requires a definition. Let $M$ be a finite-volume cusped hyperbolic $n$-manifold. Then a geodesic $\sigma$ in $M$ is called a \emph{cusp-to-cusp geodesic} if it is the image in $M$ of a geodesic in $\widetilde{M} = \Hy^n$ connecting two parabolic fixed points for the action of $\pi_1(M)$ on $\Hy^n$.

\begin{lemma}\label{lem:FiniteVolBoundary}
Let $M$ be a finite-volume hyperbolic $n$-manifold, $\Sigma$ a finite-volume embedded totally geodesic hypersurface, and $N$ a finite-volume immersed totally geodesic $m$-manifold, $m \ge 2$. Suppose $N$ is not contained in $\Sigma$. Then $N \cap \Sigma$ is a union of complete hyperbolic $(m-1)$-manifolds. If either $m \ge 3$ or $\Sigma$ is compact, then each element of this union has finite volume. If $m=2$ and $N \cap \Sigma$ is not finite-volume, then it is a union of closed geodesics and cusp-to-cusp geodesics.
\end{lemma}

\begin{proof}
Let $M$, $\Sigma$, and $N$ be as in the statement of the lemma, and let $X = N \cap \Sigma$. This intersection is necessarily transverse, since $N$ is not contained in $\Sigma$ and for any $x \in \Hy^n$ there is a unique $\Hy^\ell$ tangent to any $\ell$-plane in $T_x \Hy^n$. Being either complete or totally geodesic is preserved under intersections, so each connected component of $X$ is a properly embedded complete immersed totally geodesic submanifold. If $X$ is a union of compact components,   it is clearly finite volume and we are done.

If $X$ is not compact, then $M$ necessarily has cusps, so we want to show that each component of $X$ has cusps of full rank. In other words, we want to show that its cusps are finite-volume. Indeed, the intersection of $X$ with the compact core of $M$ is compact so this intersection clearly has finite volume. Cusp cross-sections of $X$ are intersections of cusp cross-sections of $N$ and $\Sigma$, which are closed flat $(m-1)$- and $(n-2)$-manifolds, respectively, naturally immersed inside a closed flat $(n-1)$-manifold cusp cross-section of $M$. In particular, this intersection is a closed flat $(m-2)$-manifold. When $m > 2$, this gives us a closed flat $(m-2)$-manifold cross $(0, \infty)$ as the cross-section. This has finite volume, and the lemma is proved in this case.

When $m=2$, we are in the case where our component of $X$ going out the cusp of $M$ is a complete hyperbolic $1$-manifold with cusps, i.e., a geodesic ray going out to the cusp. Since this component of $X$ is complete, it extends in the other direction to a bi-infinite geodesic contained in $N \cap \Sigma$. We claim that the other end of this geodesic also goes to a cusp of $M$. To see this, notice that otherwise such a geodesic could not be properly embedded in $M$, but this is impossible for the intersection of two complete totally geodesic subspaces of $M$. In particular, we see that this component of $X$ must be a cusp-to-cusp geodesic. This proves the lemma.
\end{proof}

\noindent We will also need the following.

\begin{lemma}\label{lem:FindElement}
Let $N$ be a finite volume hyperbolic $m$-manifold with nonempty totally geodesic boundary $\partial N$, and let $N_1, \dots ,N_k$ be the connected components of $\partial N$. If either
\begin{enumerate}

\item at least two components $N_i$ have finite $(m-1)$-dimensional volume, or

\item at least one component $N_i$ has finite $(m-1)$-dimensional volume and $N$ does not deformation retract into that boundary component,

\end{enumerate}
then there exists an element $g\in \pi_1(N)$ that is not conjugate into any of the $\pi_1(N_i)$.
\end{lemma}

\begin{proof}
We will fix an isometric embedding of the universal cover $\widetilde{N}$ into the ball model of $\Hy^m$. There is an isometric action of $\De = \pi_1(N)$ on $\widetilde{N}$, which we can extend to an isometric action on $\Hy^m$. The region $\widetilde{N}$ inside $\Hy^m$ has boundary consisting of a countably infinite collection of isometrically embedded copies of $\Hy^{m-1}$, each a lift of one of the boundary components $N_i$. We will call these the \emph{boundary hyperplanes} of the region $\widetilde N$.

All the distinct boundary hyperplanes are pairwise disjoint inside $\Hy^m$, but can have closures inside $\Hy^m \cup \partial^\infty \Hy^m$ that intersect nontrivially. Note that the closures intersect nontrivially when the corresponding point at infinity comes from distinct cusps of $\partial N$ that are asymptotic in $N$. Since we are in the ball model for $\Hy^m$, each of the boundary hyperplanes has a boundary at infinity which determines a round $(m-2)$-sphere inside the standard sphere $S^{m-1}=\partial^\infty \Hy^m$. We will call these the \emph{boundary hyperspheres}. Observe that for any $\epsilon >0$, there are only finitely many  boundary hyperspheres of diameter greater than $\epsilon$, measured in the round metric on $S^{m-1}$.

Let us consider case (1), and assume $N_1, N_2$ have finite volume. Note that the negative curvature assumption implies that $N$ cannot be a cylinder. Each of the corresponding subgroups $\De_i =\pi_1(N_i)$ inside $\De=\pi_1(N)$ is nontrivial and contains an element acting hyperbolically on $\widetilde{N}_i$. Here we are viewing $\widetilde{N}_i$ as a specific boundary hyperplane of the region $\widetilde{N}$. Let $\gamma_i\in \Delta_i$ be such an element and $\gamma_i^\pm \subset S^{m-1}=\partial^\infty \Hy^m$ denote the two limit points of the geodesic determined by $\gamma_i$. Notice that this geodesic lies inside $\widetilde{N}_i$, so the pair of points $\gamma_i^{\pm}$ lie on the boundary hypersphere $\partial^\infty \widetilde{N}_i$ inside $S^{m-1}$. Also, this geodesic is the lift of a closed geodesic in $N_i$, so the points $\gamma_i^\pm$ do not correspond to cusp points in $\partial^\infty \widetilde{N}_i$, and hence cannot lie in any other boundary hyperspheres.

Let $\Lambda \subset S^{m-1}$ be the limit set for the $\De$-action. The pair of distinct points $\gamma_1^+, \gamma_2^+$ lie in $\Lambda$ and
are a positive distance $\delta$ apart. Since they are disjoint, there are only finitely many boundary hyperspheres of diameter greater than $\delta/3$, so we can choose an $\epsilon <\delta/3$ with the property that the open $\epsilon$-balls $B_i \subset S^{m-1}$ centered at $\gamma_i^+$ do not intersect this finite collection of boundary hyperspheres, except potentially $\partial^\infty \widetilde{N}_1$ and $\partial^\infty \widetilde{N}_2$. Then $B_1\cap \Lambda$ and $B_2\cap \Lambda$ are a pair of disjoint open subsets in the limit set $\Lambda$, so by [17, \S\S 8.2F-G], there exists a hyperbolic element $g\in \De$
with $g^+\in B_1$ and $g^-\in B_2$. From our choice of $\epsilon$, the pair of points $g^\pm$ cannot lie on any single boundary hypersphere. This implies that the element $g$ cannot be conjugated into any $\pi_1(N_i)$, as desired.

Now consider case (2), and assume $N_1$ has finite volume and $N$ does not deformation retract into $N_1$. Since $N$, $N_1$ are aspherical, this implies that there exists an $h\in \pi_1(N)\smallsetminus \pi_1(N_1)$. We now consider the two distinct boundary hyperplanes $\widetilde{N}_1$ and $h\, \widetilde{N}_1$. Since $N_1$ has finite volume, we can pick a hyperbolic element $g_1\in \pi_1(N_1)$, and then define another hyperbolic element $g_2 = hg_1h^{-1}$. These stabilize the boundary hyperplanes $\widetilde{N}_1$ and $h\, \widetilde{N}_1$ respectively, so we obtain pairs of limit points $g_1^\pm \in \partial ^\infty \widetilde{N}_1$ and $g_2^\pm \in \partial ^\infty (h\, \widetilde{N}_1)$ in the corresponding boundary hyperspheres. We can then proceed as in case (1) to produce the desired element $g$. This completes the proof of the lemma.
\end{proof}

\noindent We are now ready to state and prove the main result of this section.

\begin{thm}[Closure Rigidity]\label{thm:ClosureRigidity}
Let $M = \Hy^n / \Ga$ be a finite-volume arithmetic hyperbolic $n$-manifold, $n \ge 3$, $\pi : \Hy^n \to M$ be the universal covering, and $\Sigma = \Hy^{n-1} / \Ga_0 \subset M$ be an embedded oriented finite volume totally geodesic hypersurface. For $1< m \le n-1$, suppose that $f : \Hy^m \to \Hy^n$ is a totally geodesic embedding such that the image $(\pi \circ f)(\Hy^m) \subset M$ contains a connected finite-volume $m$-manifold with totally geodesic boundary $N$ such that $\partial N \cap \Sigma \neq \emptyset$ and $N$ does not deformation retract into $\Sigma$. Then $(\pi \circ f)(\Hy^m)$ is an immersed finite-volume totally geodesic submanifold of $M$ of dimension $m$.
\end{thm}

\begin{proof}
Since $M$ contains a finite volume hypersurface, we have that $M$ is defined by a quadratic form $q$ on the vector space $V$ defined over a number field $K$ and $\Sigma$ comes from a $K$-defined codimension one subspace $V_0 \subset V$. It will be convenient to set $k = n - m$ as the codimension of $f(\Hy^m)$ in $\Hy^n$.

If $\Sigma$ is nonseparating then let $M_1=M$ cut open along $\Sigma$ and otherwise suppose that $\Sigma$ divides $M$ into two submanifolds with boundary $M_1 \cup M_2$, $\partial M_1 = \partial M_2 = \Sigma$. Set $\Ga_i = \pi_1(M_i)$ and $\Ga_0 = \pi_1(\Sigma)$, considered as a subgroup of $\Ga$. Without loss of generality, we can assume that $N \subset M_1$. Indeed, if every such component of $N$ deformation retracts into $\Sigma$, then $N$ itself would deformation retract into $\Sigma$. Then $\partial N$ contains a nonempty union $N_1, \dots, N_r$ of immersed totally geodesic submanifolds of $\Sigma$ (note that the $N_i$ might intersect if $N$ is not embedded).

We first consider the case where the $N_i$ satisfy either (1) or (2) in Lemma \ref{lem:FindElement}. Let $\De = \pi_1(N) \subset \Ga_1$ and $\De_i = \pi_1(N_i) \subset \Ga_0$. After reordering the boundary components, we can assume that $\pi_1(N_1)$ is nontrivial. Lemma \ref{lem:FindElement} implies that there is some $g \in \De$ that is not $\De$-conjugate into some $\De_i$ (i.e., does not deformation retract into $\Sigma$). For example, this is necessarily the case when $N$ has a unique boundary component, since $N$ cannot deformation retract into $\Sigma$ (see Figure \ref{fig:ConnectedBoundary}). Recall from the proof of Lemma \ref{lem:FindElement} that $N$ cannot be a cylinder.
\begin{figure}
\begin{center}
\begin{tikzpicture}
\draw[very thick, rounded corners] (4,4) -- (0,4) -- (0,0) -- (4,0);
\draw[very thick, rounded corners] (0,3.5) -- (0,4) -- (1,5) -- (5,5) -- (5, 4.8);
\draw[blue, rounded corners] (4.5, 2) .. controls (4.25, 1.75) .. (2, 1.85) -- (2,2.15) .. controls (4.25, 2.25) .. (4.5, 2);
\filldraw (4.5, 2) circle (0.025);
\draw[thick, green, rounded corners] (4.25, 1.5) -- (4.75, 2) -- (4.75, 3) -- (4.25, 2.5) -- (4.25, 1.5);
\draw[thick, green, rounded corners] (4.25, 1.5) -- (1.25, 1.5) -- (1.25, 2.5) -- (4.25, 2.5) -- (4.25, 2.25);
\draw[thick, green, rounded corners] (1.25, 2.25) -- (1.25, 2.5) -- (1.75,3) -- (4.75, 3) -- (4.75, 2.75);
\draw[very thick, red, rounded corners] (4,0) -- (4,4) -- (5,5) -- (5,1) -- (4,0);
\draw (0.5, 0.5) node {$M_1$};
\draw (5.4,4) node [red] {$\Sigma$};
\draw (4.7,1.4) node [green] {$N_1$};
\draw (1.5,3.1) node [green] {$N$};
\draw (1.8,2) node [blue] {$g$};
\end{tikzpicture}
\caption{The case when $N$ has one boundary component.}\label{fig:ConnectedBoundary}
\end{center}
\end{figure}
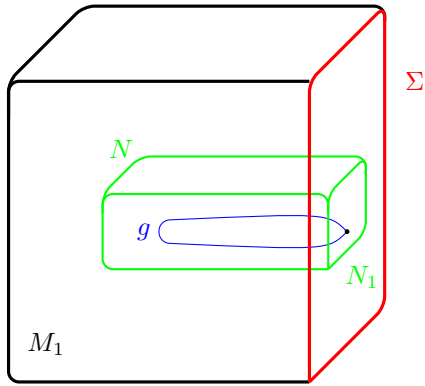
We now need some arithmetic notation.

Since $N_1$ is a finite-volume totally geodesic submanifold of $\Sigma$, it is associated with a $K$-defined subspace $V_1 \subset V_0$ of codimension $k$. Similarly, associated with $f : \Hy^m \to \Hy^n$ is a codimension $k$ subspace $W(\R) \subset V(\R)$ with the property that $V_1(\R) \subset W(\R)$. It follows that $W(\R)$ is a $\De$-invariant subspace of $V(\R)$, since the image of $f$ contains a lift of the map on universal coverings $\widetilde{N} \to \widetilde{M}_1$. In fact, there exists a vector $w \in V(\R)$ so that $W(\R)$ is spanned by $w$ and $V_1(K)$. We will show that we can take $w \in V(K)$, hence $W$ is a $K$-defined subspace of $V$ and $(\pi \circ f)(\Hy^m)$ then must determine an immersed finite-volume totally geodesic submanifold of $M$.

We have that $g(W(\R)) = W(\R)$, but $g(V_1(\R)) \neq V_1(\R)$, since $g$ is not contained in $\Ga_0$. However, the fact that $g \in \Ga_1$ implies that $g(V(K)) = V(K)$. In particular, $g(V_1(K)) \subset W(\R)$ contains a vector $v \in V(K)$ that is not contained in $V_1(\R)$. Therefore, the $K$-span of $V_1(K)$ and $v$ is a codimension $1$ subspace of $V(K)$ contained in $W(\R)$. Thus $W$ is defined over $K$ as claimed. This completes the proof when the $N_i$ satisfy either (1) or (2) in Lemma \ref{lem:FindElement}.

\medskip

By Lemma \ref{lem:FiniteVolBoundary}, what remains is the case where $m=2$ and $N \cap \Sigma$ is a finite union of closed geodesics and cusp-to-cusp geodesics. (Lemma \ref{lem:FindElement} actually allows us to reduce to the case where $N\cap\Sigma$ contains at most one closed geodesic, but this is not needed in what follows.)

Let $\sigma$ and $\tau$ be two of the boundary components of $N$. Each is either a closed geodesic in $M$, or a cusp-to-cusp geodesic. Choose a lift of $\Sigma$ to $\widetilde{\Sigma} = \Hy^{n-1} \subset \Hy^n$. We then have a lift of $N$ to $\Hy^n$ that we can complete to $\widehat{N} = \Hy^2 \subset \Hy^n$ such that $\widehat{N} \cap \widetilde{\Sigma}$ is a geodesic $\widetilde{\sigma}$ lifting $\sigma$. Associated with $\widetilde{\Sigma}$ is a subspace $V_0 \subset V$ of codimension one on which the restriction of $q$ has signature $(n-1,1)$.

First, suppose that $\sigma$ is a cusp-to-cusp geodesic. Then $\widetilde{\sigma}$ is a geodesic connecting two cusp points $z_1, z_2 \in \partial^\infty \Hy^n$ for the action of $\Ga$. Each $z_i$ is determined by a unique $K$-defined isotropic line $L_i \subset V$. Let $Y$ be the plane in $V$ spanned over $K$ by these two lines. Since $G$ has $\Q$-rank one, there cannot be $q$-orthogonal isotropic lines in $V$, hence the restriction of $q$ to $Y^\perp$ is definite. Having dimension $n-1$, we see that $q$ is positive definite on $Y^\perp$, and hence $q$ has signature $(1,1)$ on $Y$ (for the chosen embedding of $K$ in $\R$). This subspace of $V(\R)$ is associated with the geodesic in $\Hy^n$ connecting the two ideal points $z_1$ and $z_2$. Note that $Y \subset V_0$.

When $\sigma$ is a closed geodesic, we similarly obtain a $K$-defined subspace $Y$ on which $q$ has signature $(1,1)$ associated with the lift $\widetilde{\sigma}$. Our choice of lift to $\widetilde{M}$ determines a unique cyclic subgroup $C$ in the conjugacy class of cyclic subgroups of $\Ga$ associated with the geodesic $\sigma$. Then $C$ preserves a $K$-defined $q$-orthogonal splitting of $V$ into $Y \subset V_0$ and $Y^\perp$ having the same properties as in the previous case.

Now, consider the lift $\widetilde{\tau}$ of $\tau$ to $\widehat{N}$. The same arguments as above show that this determines two $K$-defined and $q$-orthogonal subspaces $Z$ and $Z^\perp$ such that the restriction of $q$ to $Z$ has signature $(1,1)$. However, since $\widetilde{\tau}$ is not contained in $\widetilde{\Sigma}$ for this choice of lifts, we see that $Z$ is not a subspace of $V_0$.

\medskip

Now, let $U$ be the subspace of $V$ spanned over $K$ by $Y$ and $Z$. Since $\widetilde{\sigma}$ and $\widetilde{\tau}$ both lie in $\widehat{N}$, i.e., the same embedded $\Hy^2$ in $\Hy^n$, it follows that $U$ is $3$-dimensional and that the restriction of $q$ to $U$ has signature $(2,1)$. Since $U$ is $K$-defined, we see that there is an arithmetic subgroup of $\Ga$ acting on $\widehat{N}$ with finite-volume quotient. This completes the proof of the theorem in this case.
\end{proof}

\begin{rem}\label{rem:ClosureFail}
The assumption that $1 < m$ is critical. Indeed, in the case where $N$ is a geodesic arc in $M_1$ connecting two distinct points in $\Sigma$, the proof completely falls apart in that $N$ and $\partial N$ have trivial fundamental group. However, the first step of the proof does work in the following sense. Suppose that $\Ga$ is an arithmetic Kleinian group that contains arithmetic Fuchsian subgroups. If $\ga_1, \ga_2$ are two purely hyperbolic elements of $\Ga$ that mutually preserve a totally geodesic hyperbolic plane in $\Hy^3$, then $\langle \ga_1, \ga_2\rangle$ is contained in an arithmetic Fuchsian subgroup of $\Ga$.
\end{rem}

\begin{rem}\label{rem:ClosureAlmostDied}
We note that the proof required extra care in the case $m=2$ with $N$ noncompact. For example, one is faced with the possibility that $N$ is an ideal polygon meeting $\Sigma$ along its sides. In this case, there is no fundamental group for $N$ or $\partial N$, but the cusp-to-cusp geodesics still carry sufficient arithmetic data for us to conclude that $N$ is contained in a totally geodesic surface.
\end{rem}

\section{Angle Rigidity}\label{sec:Angle}

\noindent
In this section we prove the following.

\begin{thm}[Angle Rigidity]\label{thm:GPSFiniteness}
Fix a non-arithmetic hyperbolic manifold $M$ that is built from building blocks, and suppose that two adjacent building blocks $N_1$ and $N_2$ are arithmetic and dissimilar. Let $A\subset M$ be a connected finite-volume immersed totally geodesic submanifold of dimension at least $2$ such that $A$ intersects the interior of $N_1$ and $N_2$, i.e. crosses a cutting hypersurface $\Sigma$. Then $A$ meets $\Sigma$ orthogonally.
\end{thm}
\noindent As remarked in the introduction, a special case of this result was recently and independently obtained in \cite{BenoistOh}.
\begin{proof}
The basic idea of the proof is as follows. Each $N_i$ is associated with an arithmetic manifold $M_i$. That the $M_i$ share isometric codimension $1$ submanifolds allows us to assume they are defined via quadratic forms over the same field $K$ \cite[\S 6]{Meyer}. Closure Rigidity allows us to assume that $A$ is built from gluing submanifolds $A_i$ of $N_i$, where $A_i = N_i \cap B_i$ for $B_i$ a finite-volume immersed totally geodesic submanifold of $M_i$. We will show when $M_1$ is not commensurable with $M_2$ that such a matching can only occur when the $B_i$ meet $\Sigma$ orthogonally. See Figure \ref{fig:Angle}. The key point is that such submanifolds $B_i$ arise from certain $K$-defined subspaces of the $K$-vector space for $M_i$, but the identification between half-spaces in $\Hy^n$ arising from gluing $N_2$ to $N_1$ is defined by a map of real vector spaces that preserves definition over $K$ only for very special subspaces.

\medskip

We first carefully explain why Closure Rigidity applies. Set $A_i = N_i \cap A \subset M_i$. This defines an immersed totally geodesic submanifold of $M_i$ with totally geodesic boundary that meets the cutting hypersurface $\Sigma$ nontrivially. If $f_i : \Hy^m \hookrightarrow \Hy^n$ denotes the extension to $\Hy^m$ of the lift $\widetilde{A}_i \to \widetilde{M}_i$, we see that $\widetilde{A}_i$ satisfies the assumptions of Theorem \ref{thm:ClosureRigidity}. The existence of the totally geodesic submanifold $B_i$ as above follows.

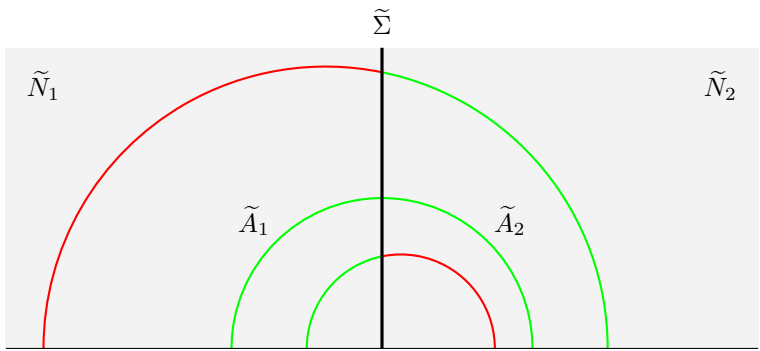
\begin{figure}
\begin{tikzpicture}
\filldraw [black, opacity = 0.001, fill opacity = 0.05] (0,0) -- (5,0) -- (5,4) -- (0,4) -- (0,0);
\node at (4.5,3.5) {$\widetilde{N}_2$};
\filldraw [black, opacity = 0.001, fill opacity = 0.05] (0,0) -- (-5,0) -- (-5,4) -- (0,4) -- (0,0);
\node at (-4.5,3.5) {$\widetilde{N}_1$};
\node at (0, 4.35) {$\widetilde{\Sigma}$};
\draw [thick, green] (-2,0) arc (180:0:2);
\draw [thick, green] (3, 0) arc (0:78.5:3.75);
\draw [thick, red] (-4.5, 0) arc (180:78.5:3.75);
\draw [thick, green] (-1, 0) arc (180:101:1.25);
\draw [thick, red] (1.5, 0) arc (00:101:1.25);
\node at (-1.7,1.7) {$\widetilde{A}_1$};
\node at (1.7,1.7) {$\widetilde{A}_2$};
\draw (-5,0) -- (5,0);
\draw [very thick] (0,0) -- (0,4);
\end{tikzpicture}
\caption{The proof of Angle Rigidity. Subspaces of $\widetilde{N}_i$ associated with geodesic subspaces of $M_i$ (green) are matched only if they meet $\Sigma$ orthogonally.}\label{fig:Angle}
\end{figure}

\medskip

Next we set up some notation and make some simplifying assumptions. Suppose that $N_1$ and $N_2$ are dissimilar arithmetic building blocks glued together by an orientation-reversing isometry of the connected totally geodesic boundary component $\Sigma$. Being arithmetic, each $N_i$ has an associated quadratic space $(V_i, q_i)$, which is well defined up to similarity. The isometry $\varphi : \Sigma \to \Sigma$ determines a similarity between codimension one subspaces $H_i$ of $V_i$ on which the restriction of $q_i$ has signature $(n-1, 1)$. We can change the similarity class of $q_2$ and assume that the similarity $H_2 \to H_1$ is an isometry.

Therefore, we can assume there is a quadratic space $(H, q)$ of dimension $n$ and signature $(n-1, 1)$ that is isometric to a codimension one subspace of each $(V_i, q_i)$. Using Gram--Schmidt, we can write
\[
(V_i, q_i) = (V_i,\langle \alpha_i \rangle \oplus q),
\]
for some $\alpha_i \in K^*$. More specifically, we can choose coordinates $\{x_0, \dots, x_n\}$ on $V_1$ and $\{y_0, \dots, y_n\}$ on $V_2$ such that
\begin{align*}
q_1(x_0, \dots, x_n) &= \alpha_1 x_0^2 + q(x_1, \dots, x_n), \\
q_2(y_0, \dots, y_n) &= \alpha_2 y_0^2 + q(y_1, \dots, y_n).
\end{align*}
The isometry between the models $\Hy_{q_i}^n$ of hyperbolic space associated with $q_2$ and $q_1$ is then induced by the map $V_2(\R) \to V_1(\R)$ defined by $y_0 \mapsto \sqrt{\al_2 / \al_1}\, x_0$ and $y_i \mapsto x_i$ for $1 \le i \le n$ (compare with \cite[\S 2.9]{GPS}). Let $\Phi$ denote this map, and note that $\Phi$ determines a map of $K$-vector spaces $V_2 \to V_1$ if and only if $\alpha_2 / \alpha_1$ is a square in $K^*$. By Witt cancellation, this holds if and only if $q_1$ is isometric to $q_2$, which is definitely not the case if $M_1$ and $M_2$ are noncommensurable.

\medskip

Consider a codimension $k$ totally geodesic submanifold $B_i$ of the arithmetic manifold $M_i$ associated with $q_i$. Then $B_i$ arises from a codimension $k$ subspace $W_i$ of the $K$-vector space $V_i$ on which the restriction of $q_i$ has signature $(n-k, 1)$. Furthermore, suppose that the intersection of $B_i$ with $\Sigma$ contains a codimension $k$ totally geodesic submanifold of $\Sigma$. Then, for the correct choice of base point, $W_i$ intersects $H$ in a codimension $k$ subspace $U_i$ of $H$ on which the restriction has signature $(n-k-1,1)$. Thus $W_i$ is generated by $U_i$ and a vector not in $H$.

Turning this around, we can construct all the submanifolds $B_i$ as above by starting with a codimension $k$ subspace $U_i$ of $H$ then taking the subspace of $V_i$ generated by $U_i$ and another vector $\xi_i$ that is not contained in $H$. Note that $\xi_i$ is not in $H$ if and only if $y_0 \neq 0$. Let $W_i$ be the span over $K$ of $\xi_i$ and $U_i$. To have $\widetilde{B}_1$ glue to $\widetilde{B}_2$ means precisely that $\Phi$ takes $W_2(\R)$ to $W_1(\R)$.

Notice that $\Phi(U_2) = U_1$, since $\Phi(H) = H$. If $\xi_2$ has coordinates $y_i$, then
\[
\Phi(\xi_2) = \left(\sqrt{\al_2 / \al_1} y_0 \,,\, y_1 \,,\, \dots, y_n\right) \in V_1(\R).
\]
If $\al_2 / \al_1$ fails to be a square in $K^*$, we see that $\Phi(W_2(\R))$ is $W_1(\R)$ for $W_1$ a $K$-defined subspace of $V_1$ if and only if $y_1 = \cdots = y_n = 0$. This proves that such a $K$-defined subspace of $V_2$ maps to a $K$-defined subspace of $V_1$ if and only if $W_i(\R)$ is the $\R$-span of $U_i$ and the vector $(1, 0, \dots, 0)$. The same argument shows that this again holds with the roles of $V_1$ and $V_2$ reversed.

\medskip

Geometrically, this proves that the submanifolds $B_i$ of $M_i$ that could possibly give rise to $A$ as in the statement of the theorem are associated with subspaces $W_i$ of $V_i$ generated by a codimension one subspace $U_i$ of $H$ along with the vector $(1,0,\dots,0)$. It remains to compute that such a $B_i$ intersects $\Sigma$ orthogonally.

The subspace of $V_i$ associated with $\Sigma$ is the $q_i$-orthogonal complement to $e = (1,0,\dots,0)$. The subspace associated with $B_i$ is the orthogonal complement $W_i$ to a subspace $Z_i$ of $V_i$ with dimension $k$. Since $(1,0,\dots,0) \in W_i$, $Z_i$ must be a subspace of $H$. However, the angle $\theta_i$ between $B_i$ and $\Sigma$ satisfies
\[
\cos^2 \theta_i = \sup_{z \in Z_i} \frac{\langle e, z \rangle_{q_i}^2}{q_i(e) q_i(z)}.
\]
See \cite[p.\ 71]{Ratcliffe}. Since $Z_i \subseteq H$, $z$ is $q_i$-orthogonal to $e$ for all $z \in Z_i$, and we see that the above expression is identically zero. This shows that $B_i$ meets $\Sigma$ orthogonally, as desired.
\end{proof}
}

\begin{rem}
Note that the proof of Theorem \ref{thm:GPSFiniteness} does not appear to use the assumption that $A$ has dimension at least two. In fact, the use of this hypothesis is hiding in our application of Closure Rigidity, in the very first step of the proof. Indeed, Closure Rigidity fails when $A$ is a geodesic; see Remark \ref{rem:ClosureFail}.
\end{rem}

\begin{rem}\label{rem:BO}
With a bit more work, one can show that this approach proves a stronger result. Suppose $M_1, M_2$ are arithmetic hyperbolic $n$-manifolds containing an isometric totally geodesic hypersurface $\Sigma$. Suppose that there is a further totally geodesic subspace $Z_0 \subset \Sigma$ of codimension $1 \le k < n-1$ and a codimension $k$ totally geodesic submanifold $N_i$ of $M_i$ that meets $\Sigma$ at $Z_0$ with angle $\theta_i$. If $\theta_2 = \pm \theta_1$, then $M_2$ is commensurable with $M_1$. One uses the above decomposition of $q_1$ and $q_2$ to directly build an isometry between the quadratic spaces.
\end{rem}

\begin{rem}\label{rem:BenoistOh}
We now discuss the similarities and differences between our work and the work of Benoist--Oh, in particular with regard to the proofs of Theorem \ref{thm:SimpleMain} in dimension $3$. The angle calculation at the end of the proof of Theorem \ref{thm:GPSFiniteness} specializes in dimension $3$ to be almost identical to their Lemma 12.2 and Corollary 12.3. Corollary 12.3 is then used in exactly the same way in the proof of their Proposition 12.1 as our observation is used in our proof of Theorem \ref{thm:GPSFiniteness}. While there are superficial similarities, the rest of the proofs are in fact substantially different. In most cases of our proof of Theorem \ref{thm:ClosureRigidity}, we use geometric techniques to prove that a connected component of the intersection of a totally geodesic submanifold with an arithmetic building block has non-elementary fundamental group and hence is arithmetically defined. This appears similar to the use of \cite[Thm.\ 11.8]{BenoistOh} in their proof of \cite[Prop.\ 12.1]{BenoistOh}, but in fact the groups considered are different. They consider the stabilizer of an entire geodesic plane in the infinite cover of the arithmetic manifold corresponding to the fundamental group of the building block and use dynamics of unipotent flows to show this group is non-elementary. The group they consider always contains the group we consider. This difference is most striking in those special cases where the object we consider may have no fundamental group while they prove the object they consider always has non-elementary fundamental group.  The proof of \cite[Prop.\ 12.1]{BenoistOh} in that paper only covers the case of three manifolds built by gluing along cocompact cutting surfaces.
\end{rem}

\section{Equidistribution, homogeneous dynamics, and proofs of the main theorems}
\label{section:nimish}

\noindent In this section, we complete the proof of Theorem \ref{thm:Main}.
We will first give some geometric reductions in Section \ref{sec:GeomtoDyn} that recast the problem as a dynamical one. For the reader's convenience, we then give a simple proof in Section \ref{subsection:codim1} in the special case of closed hypersurfaces in compact manifolds. We then go on to give a proof of the general form of Theorem \ref{thm:Main}, first in the cocompact case in Section \ref{sec:GenMainThm}, and finally in the finite-volume case in Section \ref{subsection:finvolhomdyn}.

\subsection{Some geometric preliminaries}\label{sec:GeomtoDyn}

For a Riemannian manifold $M$, we denote by $\OF(M)$ the bundle of \emph{oriented} orthonormal frames on $M$ and for any $1\le m\le n$ we denote by $G^m M$ the bundle over $M$ whose fiber over a point $p$ is the Grassmannian of $m$-dimensional subspaces in $T_pM$. For each $m$, there is a natural bundle map $\OF(M)\rightarrow G^m M$ obtained by sending an orthonormal frame $(v_1, \ldots, v_n)$ at a point $p$ to the $m$-dimensional subspace spanned by $v_1, \ldots ,v_m$. The fiber of the bundle map $\OF(M)\rightarrow G^m M$ can naturally be identified with
\[
\mathrm{S}\left(\Ort(m)\times \Ort(n-m)\right) = \left\{(A,B) \in \Ort(m)\times \Ort(n-m)\ :\ \det(A) \det(B) = 1 \right\}.
\]
An immersion $N\rightarrow M$ with $\dim(N) = m$ induces a map $G^m N \to G^m M$.

\medskip

\noindent
We now spend the rest of this subsection proving the following proposition.

\begin{prop}\label{prop:boundaryprop}
Let $M$ be an $n$-dimensional hyperbolic manifold built from building blocks containing two adjacent, arithmetic, and dissimilar building blocks. Then for any $2 \le m < n$ there is an open subset $\Omega_m \subseteq G^m M$ such that $G^m N \cap \Omega_m = \emptyset$ for any finite-volume immersed totally geodesic submanifold $N \subset M$ of dimension $m$.
\end{prop}

\begin{proof}
Fix $m \ge 2$. By assumption, there is a totally geodesic cutting hypersurface $\Sigma \subset M$ contained in the common boundary of the arithmetic building blocks.
We know from Theorem \ref{thm:GPSFiniteness} that any closed immersed totally geodesic $m$-dimensional submanifold $N\rightarrow M$ intersects $\Sigma$ orthogonally. That the submanifold $N$ intersects $\Sigma$ orthogonally means that $(T_xN)^\perp \subseteq T_x\Sigma$ for every $x\in N\cap \Sigma$ (note that this condition is in fact symmetric).

\medskip

\begin{figure}
\begin{tikzpicture}
\node at (0, 4.35) {$\gamma$};
\node at (1.1, 2) {$\epsilon_0$};
\node at (1.7, 1.4142) {$A$};
\node at (-0.3, 2) {$B$};
\node at (-0.3, 2) {$B$};
\node at (0.6,1.51) [blue] {$\overline{AB}$};
\node at (1.7,3) [red] {$\alpha$};
\draw (1.4142, 1.8) arc (90:150:0.3);
\draw [very thick, blue] (0,2) arc (90:45:2);
\draw (-5,0) -- (5,0);
\draw [very thick] (0,0) -- (0,4);
\draw [very thick, red] (1.4142,1.4142) -- (1.4142,4);
\filldraw (1.4142,1.4142) circle (0.05);
\filldraw (0,2) circle (0.05);
\end{tikzpicture}
\caption{Choosing $\epsilon_0$ in the proof of Proposition \ref{prop:boundaryprop}.}\label{fig:PickEp}
\end{figure}
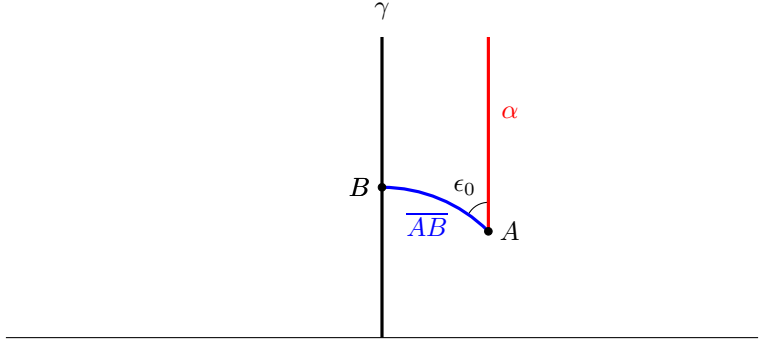

Choose a compact core $C$ for $M$ \cite[p.\ 156]{Marden}. Since $\Sigma \cap C$ is a smoothly embedded compact codimension one submanifold in $M$,
there is a small enough $\delta$ so that the exponential map, defined on the $\delta$-neighborhood of the zero section of the normal bundle to $\Sigma \cap C \subset M$,
gives a diffeomorphism from $(\Sigma \cap C) \times (-\delta, \delta)$ into a subset of $M$. For a subset $S\subset \Sigma \cap C$, denote by $\widehat S$ the image of
$S\times (-\delta, \delta)$ under this exponential map. For all points $p$ inside $\widehat S$, there is a unique minimal length geodesic
$\gamma_p$ from $p$ to $\Sigma \cap C$, which will have length $<\delta$, and will terminate at a point on $S$. Now pick $x$ in the interior of $\Sigma \cap C$,
and fix a small open metric neighborhood $U$ of $x$ in $\Sigma \cap C$, of radius smaller than the distance from $x$ to $\partial C$.

We form the subset
\[
\calV = \{ (p, W)\in G^m M \ | \ p\in { \widehat U} \,,\, \dot{\gamma}_p \in W \},
\]
and denote by $\calV_p$ the subset of $\calV$ lying above a point $p\in \widehat U$. Note that $\calV \to \widehat U$ is a fiber bundle over an open subset of $M$ with fiber $\calV_p$ a closed subset of the corresponding Grassmannian fiber of $G^m M$. Indeed, $\calV_p$ is a copy of $Gr(m-1, \mathbb R^{n-1})$ lying within the $Gr(m, \mathbb R^n)$ fiber of $G^m M$.

Take a geodesic $\gamma$ in $\Hy^2$, a point $A$ at distance $\delta/2$ from $\gamma$, and let $B$ denote the projection of $A$ onto $\gamma$. Set $0<\epsilon_0< \pi/2$ to be the angle between the geodesic $\overline{AB}$ from $A$ to $B$ and the geodesic ray $\alpha$ from $A$ to one of the endpoints of $\gamma$ on $\partial^\infty \Hy^2$. See Figure \ref{fig:PickEp}. In particular, notice that any geodesic arc in $\Hy^2$ starting at $A$ with initial angle $0 < \epsilon < \epsilon_0$ with $\overline{AB}$ then must intersect $\gamma$ in $\Hy^2$.

Take $0 < \epsilon < \epsilon_0$, and consider the set $N_\epsilon(\calV)$ defined by taking the fiber-wise $\epsilon$-neighborhood of the set $\calV$, measured in the angular metric on the Grassmannian. Thus a pair $(p, W)$ lies in $N_\epsilon(\calV)$ if and only if $p\in \widehat U$ and there exists a $W' \in \calV_p$ such that $\theta (W, W')<\epsilon$. Finally, we let
\[
\Omega_m = N_\epsilon (\calV)\ssm \calV,
\]
and note that $\Omega_m$ is an open subset of $G^m M$.

We now check that $G^m N \cap \Omega_m = \emptyset$ for any finite-volume immersed $m$-dimensional totally geodesic submanifold $N$ of $M$. By way of contradiction, assume that there exists a point $(p, W) \in G^m N \cap \Omega_m$. Since $N$ has dimension $m$, this means that $p\in N$, $W= T_pN \subset T_pM$, and $0<\theta (W, \calV_p) < \epsilon$. This last inequality tells us that $0<\theta (W,\dot{\gamma}_p)< \epsilon$, where $\gamma_p$ is the unique geodesic from $p \in \widehat U$ to $U \subset \Sigma$.

We let $w\in W$ denote the projection of $\dot{\gamma}_p$ to $T_pN$, and consider the subspace $Z=\spn\{w, \dot{\gamma}_p\}$ of $T_pM$. We have that $\dim(Z\cap T_pN) = \dim(Z\cap W) =1$, and hence $Z$ defines a (not necessarily closed) geodesic $\eta$ contained in the submanifold $N$. It also contains the direction vector $\dot{\gamma}_p$ corresponding to the minimal geodesic from $p$ to $\Sigma$. Moreover, we have that
\[
z=\dot{\gamma}_p - w \in (T_pN)^\perp,
\]
is a nonzero vector, hence $\dim (Z\cap (T_pN)^\perp)=1$.

\medskip

The $2$-dimensional subspace $Z$ gives rise to an isometric immersion $f:\mathbb H^2 \rightarrow M$ whose image contains both geodesics $\gamma_p$ and $\eta$. We let $q$ denote the terminal point of $\gamma_p$ in the open subset $U$ of the hypersurface $\Sigma$.
Since $\theta (\dot{\gamma}_p, \dot{\eta})=\theta (\dot{\gamma}_p , w) <\epsilon$, our choice of $\epsilon<\epsilon_0$ implies that the geodesic
$\eta$ intersects $\Sigma$ at a point $q'$. Since $\dot{\eta} \in T_pN$, we see that $\eta$ is contained entirely in the submanifold $N$, and hence that $q'\in N\cap \Sigma$.

Notice that we now have a geodesic triangle contained in the isometrically immersed hyperbolic plane $f(\Hy^2)$, consisting of the geodesic segment $\gamma_p$, the geodesic segment $\eta$, and the geodesic $\xi$ joining the two endpoints $q, q' \in \Sigma$. See Figure \ref{fig:hyptriangle}. Since $\Sigma$ is totally geodesic, we see that $\xi$ is contained entirely
in $\Sigma$, and basic hyperbolic trigonometry implies that $\theta_{q'}(\eta, \xi) < \pi/2$, since $\theta_q(\xi, \gamma_p) =\pi/2$.

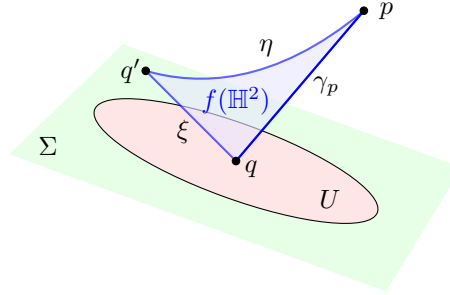
\begin{figure}
\begin{tikzpicture}
\filldraw [fill=green!10, green!10] (2,-1.7279404685324) -- (-3,0.091910702798600) -- (-1.5,1.5459553513993) -- (3, -0.091910702798600);
\filldraw [fill = red!10, rotate around={160:(0,0)}] (0,0) ellipse (2 and 0.5);
\filldraw [thick, blue, fill=blue!10, opacity=0.6] (1.7,2) -- (0,0) -- (-1.2,1.2) .. controls (-0.5,1) and (0.5,1) .. (1.7, 2);
\draw [thick, blue] (1.7,2) -- (0,0);
\node at (1.25,-0.5) {$U$};
\filldraw (0,0) circle (0.05);
\node at (0.2,-0.1) {$q$};
\filldraw (-1.2,1.2) circle (0.05);
\node at (-1.4,1.2) {$q'$};
\node at (-2.5,0.2) {$\Sigma$};
\filldraw (1.7,2) circle (0.05);
\node at (2,2) {$p$};
\node at (1.2,1) {$\ga_p$};
\node at (0.4,1.5) {$\eta$};
\node at (-0.7,0.4) {$\xi$};
\node at (0,0.8) [blue] {$f(\Hy^2)$};
\end{tikzpicture}
\caption{The geodesic triangle in Proposition \ref{prop:boundaryprop}.}\label{fig:hyptriangle}
\end{figure}

We claim that $N$ fails to be orthogonal to $\Sigma$ at the point $q'$. To see this, we need to exhibit a vector in $(T_{q'}N)^\perp$ that does not lie in the subspace $T_{q'}\Sigma$. Recall that $z\in Z$ was a vector orthogonal to $T_pN$, and tangent to the immersed totally geodesic hyperbolic plane $f(\Hy^2)$. We can thus parallel transport this orthogonal vector along $\eta \subset N$ to obtain a vector $z' \in T_{q'}M$. Since parallel transport preserves orthogonality, we see that $z' \in (T_{q'}N)^\perp$. Also, $z\in T_pf(\Hy^2)$, and $\eta$ is contained in $f(\Hy^2)$, which forces $z' \in T_{q'}f(\Hy^2)$. Since we have $z \perp \dot{\eta}(p)$, it follows that $z' \perp \dot{\eta}(q')$.

Finally, focusing on the behavior at $T_{q'}M$, we have three vectors that all lie in the $2$-dimensional subspace $T_{q'}f(\Hy^2) \subset T_{q'}M$: the vector $\dot{\xi}$, which spans the $1$-dimensional subspace
\[
T_{q'}f(\Hy^2) \cap T_{q'}\Sigma,
\]
the vector $\dot{\eta}$, which we argued forms an angle $\theta_{q'}(\dot{\eta}, \dot{\xi}) < \pi/2$ with $\dot{\xi}$, and the vector $z' \in (T_{q'}N)^\perp$ that is
orthogonal to $\dot{\eta}$. It immediately follows that $z'$ cannot be parallel to $\dot{\xi}$. Since $\dot{\xi}$ generates the intersection of $(T_{q'}N)^\perp$ with
$T_{q'}\Sigma$, we have found an element of $(T_{q'}N)^\perp$ that cannot be in $T_{q'}\Sigma$, and hence $N$ is not orthogonal to $\Sigma$.

\medskip

\noindent This contradicts Theorem \ref{thm:GPSFiniteness}, so we conclude that the set $G^m N \cap \Omega_m$ is indeed empty, completing the proof of the proposition.
\end{proof}

\subsection{Codimension 1}
\label{subsection:codim1}

In this subsection, we show how Proposition \ref{prop:boundaryprop} can be used to complete the proof of Theorem \ref{thm:Main} for the case of closed hypersurfaces in a compact manifold.
For this, we use the following result of Shah \cite[Thm.\ B]{Shah} from homogeneous dynamics.

\begin{thm}
\label{thm:nimish}
Suppose $n \ge 3$, $G=\SO_0(n,1)$, $\Gamma < G$ a cocompact lattice, and $W=\SO_0(n-1,1)$. Then every $W$ invariant subset of $G/\Gamma$ is either dense or the union of finitely many closed $W$ orbits.
\end{thm}

It is worth mentioning that one can also prove Theorem \ref{thm:nimish} by using Ratner's work on invariant measures for unipotent flows \cite{Ratner} and work of either Ratner or Dani--Margulis on equidistribution and orbit closures \cite{RatnerOrbitClosure, DaniMargulis}.
In fact we will use this approach in Section \ref{sec:GenMainThm}. The proof given in Shah's paper follows instead Margulis' original proof of the Oppenheim conjecture using topological dynamics \cite{MargulisOppenheim, DaniMargulisOppenheim}. It seems possible that those methods might prove all the results we need in this paper, but we do not pursue this here.

The following standard lemma shows how immersed totally geodesic submanifolds correspond to closed $W$ orbits in $G/\Gamma$.

\begin{lemma}
\label{lem:submanifoldstoorbits}
Let $M$ be a compact hyperbolic manifold of dimension $n$ and $N \subset M$ a closed immersed totally geodesic hypersurface. Then there are either one or two closed $W$ orbits in $G/\Gamma$ which project onto $N$. Furthermore, distinct immersed totally geodesic hypersurfaces give rise to distinct $W$ orbits.
\end{lemma}

\begin{proof}
Given a closed immersed totally geodesic hypersurface $N$, its lift $\widetilde N$ to the universal cover $\widetilde M= \Hy^n$ has stabilizer
\[
H=g^{-1}\Ort^+(n-1,1)g,
\]
where $g$ is any isometry taking $\widetilde N$ to the standard $\Hy^{n-1}$ stabilized by
\[
\Ort^+(n-1,1) = \mathrm{S}(\Ort(1) \times \Ort(n-1,1)) \cap \SO_0(n,1).
\]
Since $N$ is closed, $H \cap \Gamma$ is a cocompact lattice. Writing $\Ort^+(n-1,1)g = gH$ we see that the $\Ort^+(n-1,1)$ orbit of $[g]$ is the $g$ translate of $H/{H\cap\Gamma}$ in $G/\Gamma$. As $\Ort^+(n-1,1)$ contains $\SO_0(n-1,1)$ as a subgroup of index two, this $\Ort^+(n-1,1)$ orbit is either one or two $\SO_0(n-1,1)$ orbits giving the first claim. Since distinct choices of $N$ give rise to distinct groups $H$, the last statement is clear.
\end{proof}

\noindent We are now in a position to conclude Theorem \ref{thm:Main} in our special case.

\begin{thm}
Let $M$ be a compact hyperbolic manifold built from building blocks with two adjacent, arithmetic, and dissimilar building blocks. Then there are only finitely many closed immersed totally geodesic hypersurfaces in $M$.
\end{thm}

\begin{proof}
Denote by $\Omega_{n-1}$ the open set provided by Proposition \ref{prop:boundaryprop} applied when $m=n-1$ and by $V$ its pre-image in $G/\Gamma$
(which can be identified with $OF(M)$, the orthonormal frame bundle). By Proposition \ref{prop:boundaryprop} and Lemma \ref{lem:submanifoldstoorbits}, the $W$ orbits arising from all closed immersed totally geodesic hypersurfaces in $M$ have trivial intersection with $V$. By Theorem \ref{thm:nimish} we thus have a finite collection of $W$ orbits and hence we must only have finitely many totally geodesic closed immersed hypersurfaces in $M$.
\end{proof}

\subsection{Theorem \ref{thm:Main} in the compact case}
\label{sec:GenMainThm}

For simplicity, we write this subsection assuming that $M$ is compact, that is to say that $\Gamma <G$ is a cocompact lattice. Throughout we assume $M$ satisfies the hypotheses of Theorem \ref{thm:Main} and use the conventions that $m=n-k$ and $W_m =\SO_0(m,1)$. Our choice of $W_m$ ensures that it is connected and generated by one-parameter unipotent subgroups.

We begin by recalling some notation and a result of Dani--Margulis. Let $\calH$ be the collection of subgroups $H$ of $G$ such that $H \cap \Gamma$ is a lattice in $H$ and $\Ad(H\cap\Gamma)$ is Zariski dense in $\Ad(H)$. Given $H \in \calH$ and a group $W \subset G$ generated by unipotent elements, define
\[
X(H,W) = \{ g \in G \mid Wg \subset gH \}.
\]
We remark for later that the following is immediate from the definition.

\begin{lemma}
\label{lemma:conjugationxhw}
If $x \in G$, then $X(H,xWx^{-1})=xX(H,W)$.
\end{lemma}

We now recall a special case of \cite[Thm.\ 3]{DaniMargulis}, where the advantage of assuming that $\Gamma$ is cocompact is that we may take $K=G/\Gamma$ in their notation. (We emphasize here that our $K$ below will be a maximal compact subgroup of $G$.)

\begin{thm}
\label{thm:dm}
Let $G$ be a Lie group and $\Gamma<G$ a cocompact lattice. Let $\mu$ be the Haar measure on $G/\Gamma$. Let $U=\{u_t\}$ be a one parameter $\Ad$-unipotent subgroup of $G$ and let $\phi$ be a bounded continuous function on $G/\Gamma$. Fix $\epsilon >0$, then there exist finitely many subgroups $H_1, \ldots H_r \in \calH$ such that for $x \notin \cup_{i=1}^r X(H_i, U)\Gamma/\Gamma$ there exists a $T_0 \ge0$ such that
\[
\left|\frac{1}{T}\int_0^T \phi(u_tx) dt - \int_{G/\Gamma} \phi d\mu\right|<\epsilon,
\]
for all $T > T_0$.
\end{thm}

\noindent The theorem says that any $u_t$ trajectory not contained in any $X(H_i,U)$ eventually has time average for $\phi$ equal to the space average for $\phi$. Let $\pi_m:\OF(M)\rightarrow G^{m}M$ be the bundle map defined in \S \ref{sec:GeomtoDyn}, $\Omega_m$ be the open set provided by Proposition \ref{prop:boundaryprop}, and let $V_m = \pi_m^{-1}(\Omega_m)$.
For our purposes we will always take $\phi_m$ to be a compactly supported function of total integral $1$ with support in $V_m$. Then any unipotent trajectory in $\OF(M)$ corresponding to part of the horocycle flow of a closed immersed totally geodesic submanifold must have
\[
\frac{1}{T}\int_0^T \phi_m(u_tx) dt=0,
\]
for all $T$. In particular, the entire trajectory is contained in
\[
\bigcup_{i=1}^r X(H_i, U)\Gamma/\Gamma,
\]
by Theorem \ref{thm:dm}.

In the remainder of this section, we show that all pre-images of closed immersed totally geodesic submanifolds are contained in
\[
\bigcup_{i=1}^r X(H_i, W_m)\Gamma/\Gamma,
\]
for some finite collection of subgroups $H_i\in\calH$, then give a concrete description of each $X(H_i,W_m)$. This is the content of Lemma \ref{lemma:onetomany} and Proposition \ref{prop:xhw}.
The geometric translations in Corollaries \ref{cor:finitelymanyxhw} and \ref{cor:oneortwosubs} will then give the requisite finiteness statement.
We now begin with a key reduction for our proof, noting that $W_m$ is generated by unipotent subgroups.

\begin{lemma}
\label{lemma:onetomany}
Fix $m$ between $2$ and $n-1$, let $W=W_m$, and let $U_1, \ldots, U_s$ be unipotent subgroups generating $W$. If $H \in \calH$, then
\[
X(H,W) = \bigcap _{i=1}^s X(H, U_i).
\]
\end{lemma}

\begin{proof}
It is clear that $X(H,W)\subset X(H,U_i)$ for all $1\le i\le s$.
Taking the derivative of the defining equation of $X(H,W)$ we see that $g\in X(H,W)$ if and only if $\Ad(g^{-1}) \fw \subset \fh$.
Consequently if $g\in X(H,W)$, then $\Ad(g^{-1}) \fu \subset \fh$ for any connected subgroup $U<W$.
As $\fh$ is a subalgebra, if $\Ad(g^{-1}) \fu_i \subset \fh$ for all $1\le i\le s$ then $\Ad(g^{-1}) \fw \subset \fh$.
This gives the reverse containment.
\end{proof}

\begin{cor}
\label{cor:finitelymanyxhw}
Fix $m$ between $2$ and $n-1$, and let $K \cong \SO(n)$ denote a maximal compact subgroup of $G = \SO_0(n,1)$. Then there are finitely many subgroups $H_{1,m}, \cdots H_{j_m,m}$ of $G$ such that the set of all closed immersed totally geodesic submanifolds of $M$ of dimension $m$ is contained in
\[
\calX_m = K\backslash K\left(\bigcup_{i=1}^{j_m} X(H_{i,m},W_m)\Gamma/\Gamma\right)\subset M.
\]
In particular, the set of all closed immersed totally geodesic submanifolds of $M$ of dimension between $2$ and $n-1$ is contained in the union of the $\calX_m$.
\end{cor}

\begin{proof}
Fix some $m$ between $2$ and $n-1$.
By definition, the pre-image in $G/\Gamma$ of any closed immersed totally geodesic submanifold $N$ of $M$ with dimension $m$ has empty intersection with $V_m$. Arguing as in Lemma \ref{lem:submanifoldstoorbits}, we see that $\widetilde N \subset \widetilde M$ is stabilized by a conjugate of
\[
L = \mathrm{S}(\Ort(n-m) \times \Ort(m,1)) \cap \SO_0(n,1).
\]
Again as in Lemma \ref{lem:submanifoldstoorbits}, there is a closed $L$ orbit in $G/\Gamma$ covering $N$ which is clearly $W_m$ invariant, since $W_m<L$.
Therefore the set of all closed immersed totally geodesic submanifolds of $M$ is contained in a $W_m$ invariant set avoiding $V_m$.
Picking any bounded continuous function $\phi_m$ with support in $V_m$ and average $\int \phi_m d\mu =1$, the statement is then immediate from Theorem \ref{thm:dm} and Lemma \ref{lemma:onetomany} for totally geodesic submanifolds of dimension $m$. Repeating the argument in each dimension establishes the last statement in the corollary and completes the proof.
\end{proof}

To complete the proofs of our main theorems, we now need to compute the sets $X(H,W_m)$ for any $H\in\calH$. Note that by definition any $H$ for which $X(H_i,W_m)$ is nonempty contains a conjugate of $W_m$ or equivalently $H$ is a conjugate of a group containing $W_m$. The following lemma is elementary and contained in, for example, Einsiedler--Wirth \cite[Cor.\ 3.2]{EinsiedlerWirth}.

\begin{lemma}
\label{lemma:intermediatesubgroups}
Any closed subgroup $H<G=\SO_0(n,1)$ containing $W_m$ is of the form
\[
\mathrm{S}(K \times \Ort(\ell,1)) \cap G,
\]
where $\ell \geq m$ and $K$ is a compact subgroup of $\Ort(n-\ell)$.
\end{lemma}

\noindent We also need an elementary algebraic lemma concerning $W_m$, $H$, and $G$.

\begin{lemma}
\label{lemma:ghw}
Fix $m$ between $2$ and $n-1$, let $W=W_m$, and let $H$ be any subgroup satisfying $W<H<G$.
Then any subgroup $W'$ of $H$ isomorphic to $W$ is conjugate to $W$ in $H$. Furthermore, $\Aut(W) < G$ and $\Aut(W) <H$ unless $H$ is $K \times \SO(m,1)$ for some compact $K$.
\end{lemma}

\begin{proof}
For the first statement, we can assume $H=\SO_0(\ell,1)$ since any subgroup $W' \cong \SO_0(m,1)$ is obviously contained in a subgroup of $H$ isomorphic to $\SO_0(\ell,1)$. In this case we identify $\Isom(\Hy^\ell)$ with $H$ by writing $\Hy^\ell= K\backslash H$, where $K=\SO(\ell)$ is a maximal compact subgroup of $H$. The orbits of $W$ and $W'$ in $\Hy^\ell$ are totally geodesic copies of $\Hy^m$. Since $\SO_0(\ell, 1)$ acts transitively on totally geodesic embeddings of $\Hy^m$, there is an isometry $g \in H$ carrying the $W$ orbit to the $W'$ orbit and hence $g^{-1}Wg< W'$. The last statement follows similarly as $\Aut(W)$ is simply the group of all isometries of $\Hy^m$, including the orientation-reversing isometries, and these can be realized in $\Isom(\Hy^\ell)$.
The only case where these are not all in $H$ is when $H$ contains only $\SO_0(m,1)$ and not $\Ort^+(m,1)$.
\end{proof}

We now describe the possible $X(H,W_m)$, where in what follows we write $Z(W)$ for the centralizer of a subgroup $W$ in $G$.

\begin{prop}
\label{prop:xhw}
Fix $m$ between $2$ and $n-1$, let $W=W_m$, and fix $H\in\calH$ for which $X(H,W)$ is nonempty.
Then either there exists $x \in G$ such that $X(H,W)= Z(W)xH$ or there exists $x_1,x_2 \in G$ such that $X(H,W) = Z(W) x_1 H \cup Z(W) x_2 H$.
\end{prop}

\begin{proof}
Conjugating $W$ by a suitable element $x\in G$ we may assume that $W'=x^{-1}Wx<H$.
We first show the corresponding result for $W'$.

It is clear that $Z(W')H \subset X(H,W')$ and we need only prove the reverse inclusion.
To this end, let $y \in X(H,W')$ and define $W''=y^{-1}W'y $ to be the resulting subgroup of $H$.
Then from Lemma \ref{lemma:ghw}, we know that $W''$ is conjugate to $W'$ by an element $h\in H$. When $\Aut(W') < H$, we may moreover choose $h$ so that conjugating $W'$ by $yh$ is the identity on $W'$. This immediately implies that $yh \in Z(W')$ and hence $y\in Z(W')H$. When $\Aut(W')$ is not contained in $H$, we can only guarantee that $yh$ is either in $Z(W')$ or is of the form $zf$ where $z \in Z(W')$ and $f \in G$ is any fixed element of $G$ inducing the outer automorphism of $W$. Hence either $X(H,W')=Z(W')H$ or $X(H,W')=Z(W')H \cup Z(W')fH$ as claimed.

For the general case, we use Lemma \ref{lemma:conjugationxhw} to compute $X(H,W)$. Since
\[
Z(W')=Z(x^{-1}Wx)=x^{-1}Z(W)x,
\]
in the first case we get that $X(H,W)= x x^{-1}Z(W)xH= Z(W)xH$ and in the second case we get that $X(H,W) = Z(W)xH \cup Z(W)xfH$.
\end{proof}

\begin{cor}\label{cor:oneortwosubs}
For any fixed $H \in \calH$ and $m$ between $2$ and $n-1$, the subset $K\backslash KX(H,W_m)\Gamma/\Gamma \subset M$ is either one or two closed totally geodesic submanifolds of $M$.
\end{cor}

\begin{proof}
It is enough to see that $K\backslash K Z(W) x H\Gamma/\Gamma$ is as described. Since $Z(W) \subset K$, it can be dropped from the expression and we have $K\backslash K x H\Gamma/\Gamma$. By Lemma \ref{lemma:intermediatesubgroups},
\[
H= \mathrm{S}(C \times \Ort(\ell,1)) \cap \SO_0(n,1),
\]
where $C$ is a subgroup of $K$ and $H\cap \Gamma$ is a lattice.
Hence we are reduced to the $K, \Gamma$ double coset of $xH$ which is precisely a closed immersed totally geodesic submanifold of $M$.
\end{proof}

\noindent Theorem \ref{thm:Main} in the case where $M$ is compact is now immediate from Corollary \ref{cor:finitelymanyxhw} and Corollary \ref{cor:oneortwosubs}.

\subsection{Theorem \ref{thm:Main} in the general finite-volume case}
\label{subsection:finvolhomdyn}

The goal of this subsection is to prove an analogue of Corollary \ref{cor:finitelymanyxhw} in the case when $M$ has finite volume, after which one is reduced to repeating the proof in Section \ref{sec:GenMainThm}.
There is some difficulty in extending the argument above to the case of finite covolume, since the general version of Theorem \ref{thm:dm}, stated below, is more complicated. Our main tool for circumventing this issue is the following simple geometric lemma. We emphasize that in this subsection we use `closed' in the sense of orbits, so any immersed finite-volume totally geodesic submanifold is considered to be closed in what follows.

\begin{lemma}
\label{lemma:solvablecusps}
Let $M$ be a noncompact finite volume hyperbolic manifold of dimension $n$. Then there is a compact set $C \subset M$
such that any closed immersed totally geodesic submanifold $N$ in $M$ of dimension between $2$ and $n-1$ has nonempty intersection with $C$.
\end{lemma}

\begin{proof}
Note that each $N$ has fundamental group $\pi_1(N)$ which is a lattice in some $\SO_0(m,1)$ and injects in $\pi_1(M)$.
Moreover, the cusps of $M$ have solvable fundamental groups and consequently $\pi_1(N)$ admits no embeddings into these groups. Let $C_1$ be a compact core of $M$ \cite[p.\ 156]{Marden}, so $M\smallsetminus C_1$ is entirely contained in the cusps. Then the image of every $N$ in $M$ must intersect $C_1$. A compact subset $C$ of $M$ containing a neighborhood of $C_1$ gives the requisite compact set.
\end{proof}

We now state the general form of the theorem of Dani--Margulis \cite[Thm.\ 3]{DaniMargulis}, where in their language we assume that $K=C$, $F=\{x\}$, and the $C_i$ are all of the $X(H_i,U)$. As in the previous section, our $K$ below will be a maximal compact subgroup of $G$.

\begin{thm}
\label{thm:dm2}
Let $G$ be a connected Lie group and $\Gamma <G$ a lattice and $\mu$ be the $G$ invariant probability measure on $G/\Gamma$. Let $U=\{u_t\}$ be any $\Ad$-unipotent one parameter subgroup of $G$ and $\phi$ a bounded continuous function on $G/\Gamma$. Let $C$ be a compact subset of $G/\Gamma$ and $\epsilon >0$ be given. Then there exist finitely many proper closed subgroups $H_1, \ldots H_r$ such that $H_i \cap \Gamma$ is a lattice in $H_i$ and such that for any $x$ in
\[
C\smallsetminus\left( \bigcup_{i=1}^r (X(H_i,U)\Gamma/\Gamma\right),
\]
there exists $T_0 \geq 0$ such that
\[
\left|\frac{1}{T}\int_0^T \phi(u_tx) dt - \int_{G/\Gamma} \phi d\mu\right|<\epsilon,
\]
for all $T > T_0$.
\end{thm}

The need for Lemma \ref{lemma:solvablecusps} comes from the fact that the $H_i$ depend on $C$ in this statement. We believe this is an artifact of the proof given in \cite{DaniMargulis}, but modifying their proof to avoid this dependence requires significant changes. We now show how to combine Lemma \ref{lemma:solvablecusps} and Theorem \ref{thm:dm2} to obtain our desired finiteness result.

\begin{prop}
\label{prop:finitevolume}
Fix $m$ between $2$ and $n-1$, let $W_m=\SO_0(m,1)$, and let $M$ be any finite volume hyperbolic manifold satisfying the hypotheses of Theorem \ref{thm:Main}.
Then there are finitely many subgroups $H_{1,m}, \cdots H_{j_m,m}$ of $G$ such that the set of all closed immersed totally geodesic submanifolds of $M$ of dimension $m$ is contained in
\[
\calX_m = K\backslash K\left(\bigcup_{i=1}^{j_m} X(H_{i,m},W_m)\Gamma/\Gamma\right),
\]
where $K$ is a maximal compact subgroup of $G$. In particular, the set of all closed immersed totally geodesic submanifolds of $M$ of dimension between $2$ and $n-1$ is contained in the union of the $\calX_m$.
\end{prop}

\begin{proof}
 As in the proof of Corollary \ref{cor:finitelymanyxhw}, we see that each totally geodesic submanifold of $M$ of dimension $m$ gives rise to a closed orbit of
\[
L=\mathrm{S}(\Ort(n-m) \times \Ort(m,1)) \cap \SO_0(n,1).
\]
By Lemma \ref{lemma:solvablecusps}, there is a compact subset $C$ in $G/\Gamma$ such that any closed $L$ orbit intersects $C$. We also know that there are finitely many unipotent subgroups $U_1, \ldots, U_{s_m}$ that generate $W_m=\SO_0(m,1)<L$, and so
\[
X(H, W_m) = \bigcap_{j=1}^{s_m} X(H, U_j),
\]
for any $H\in\calH$.

We apply Theorem \ref{thm:dm2} to each $U_j$ using $C$ as the compact set, taking $\epsilon$ to be any fixed number $0<\epsilon<1$, and with $\phi_m$ a compactly supported function with integral $1$ and support in $V_m=\pi^{-1}(\Omega_m)$. Given a closed $L$ orbit $\mathcal{O}$, we pick a point $x$ contained in $\mathcal{O} \cap C$. For any $U_j$ the time average of $\phi_m$ over the $x$ orbit is zero, and so by the above there are finitely many $H_{i,j,m}$ such that $x$ is in one of the $X(H_{i,j,m}, U_j)$. Taking intersection over $j$ and re-indexing the resulting collection of subgroups completes the proof for fixed $m$. Repeating the argument in each dimension establishes the last statement in the corollary and completes the proof.
\end{proof}

\noindent Given Proposition \ref{prop:finitevolume}, the proof of Theorem \ref{thm:Main} for finite-volume manifolds proceeds exactly as in the compact case.

\section{Examples}\label{sec:Exs}

\subsection{Hyperbolic $3$-manifolds and link complements}\label{subsec:Links}{\ }

\medskip

\noindent
We begin with a few examples to give a sense of the known possibilities for existence of arithmetic/non-arithmetic hyperbolic $3$-manifolds with certain behavior for their geodesic surfaces.

\begin{ex}\label{ex:Fig8}
The figure-eight knot complement contains infinitely many distinct immersed totally geodesic surfaces. They may be taken to be closed. \cite[\S 9.6]{MaclachlanReid}
\end{ex}

\begin{ex}\label{ex:NoSurfacesA}
The Weeks manifold \cite[\S 13.5, Ex.\ 1]{MaclachlanReid} is arithmetic and contains no totally geodesic surfaces.
\end{ex}

\begin{ex}\label{ex:NoSurfacesNA}
Performing $(-4,1)$-surgery on the sister of the figure-eight knot complement \cite[\S 13.5, Ex.\ 8]{MaclachlanReid} determines a non-arithmetic closed hyperbolic manifold $M$ of volume $1.42361...$\ that satisfies the conditions of \cite[Thm.\ 5.3.1]{MaclachlanReid}, and hence contains no totally geodesic surfaces.
\end{ex}

\begin{ex}\label{ex:NoCompactSurfaces}
Consider the three-twist knot $5_2$. It has (invariant) trace field the cubic extension of $\Q$ with minimal polynomial $t^3 - t^2 + 1$. It follows from \cite[Thm.\ 5.3.8]{MaclachlanReid} that the associated hyperbolic knot complement contains no \emph{closed} totally geodesic surfaces. It does, however, contain an immersed totally geodesic $3$-punctured sphere (see \cite{ReidTG}).
\end{ex}

\begin{ex}
Calegari showed that fibered knots (more generally, fibered knots in rational homology spheres) with trace field of odd degree cannot contain immersed totally geodesic surfaces \cite{Calegari}. For instance, this holds for the knot $8_{20}$ in \cite{Rolfsen}.
\end{ex}

\noindent
As further motivation, we recall the Menasco--Reid conjecture.

\begin{conjec}[Menasco--Reid \cite{MenascoReid}]\label{conjec:MenRd}
Let $K$ be a hyperbolic knot. Then $S^3 \ssm K$ does not contain a closed embedded totally geodesic surface.
\end{conjec}

In \cite{MenascoReid}, Menasco and Reid prove the conjecture for tunnel number one knots. Our results then have the following `almost' version of the conjecture, which also allows one to promote embedded to immersed:

\begin{thm}\label{thm:MenRdKnot}
Let $K$ be a hyperbolic knot for which $S^3 \ssm K$ satisfies the conditions of Theorem \ref{thm:Main}. Then $S^3 \ssm K$ contains only finitely many closed, immersed totally geodesic surfaces.
\end{thm}

Menasco and Reid also proved that alternating hyperbolic links cannot contain embedded closed totally geodesic surfaces \cite[Thm.\ 1]{MenascoReid}. Such a link can certainly contain punctured totally geodesic surfaces, and, when Theorem \ref{thm:Main} applies, we see that $S^3 \ssm L$ can only contain finitely many such surfaces. To prove Theorem \ref{thm:LinkMain}, we must produce an infinite family of examples to which Theorem \ref{thm:Main} will apply. We do this via the operation on links called the \emph{belted sum}, which is justified by the following theorem of Adams.

\begin{thm}[Adams \cite{Adams}]\label{thm:Adams}
Let $M$ be a complete finite-volume hyperbolic $3$-manifold. Then any incompressible and properly embedded $3$-punctured sphere in $M$ is isotopic to one that is totally geodesic.
\end{thm}

Let $M_1, M_2$ be noncommensurable arithmetic hyperbolic $3$-manifolds each containing an embedded totally geodesic $3$-punctured spheres. Since the hyperbolic structure on a $3$-punctured sphere is unique, we can cut $M_1$ and $M_2$ open along these surfaces and obtain a Gromov and Piatetski-Shapiro non-arithmetic manifold. The proof of \cite[Thm.\ 4.1]{Adams} makes this connection very explicit. In particular, Theorem \ref{thm:Main} applies and the hybrid manifold has finitely many totally geodesic surfaces.

We will apply this in the case where $M_i = S^3 \ssm L_i$ for $L_i$ a link containing an unknotted component as in Figure \ref{fig:2Disk}. This contains an obvious properly embedded $2$-punctured disk $D_i$, which is then isotopic to a totally geodesic $3$-punctured sphere in $M_i$. One can then glue $M_1 \ssm D_1$ to $M_2 \ssm D_2$ to obtain a new link $L_{1,2}$ in $S^3$, which is visibly built from arithmetic building blocks. This link is the belted sum of $L_1$ and $L_2$.

\definecolor{linkcolor0}{rgb}{0.85, 0.15, 0.15}
\definecolor{linkcolor1}{rgb}{0.15, 0.15, 0.85}
\begin{figure}
\centering
\begin{tikzpicture}[line width = 1.5, line cap=round, line join=round,scale=0.3]
\draw [yellow, fill=yellow, fill opacity=0.2, draw opacity=0.01] (2.49, 15.27) .. controls (2.49, 15.27) and (2.49, 12.37) .. (2.49, 12.37) .. controls (2.08, 12.77) and (2.08, 12.77) .. (1.70, 12.37) .. controls (1.70, 12.37) and (1.70, 15.27) .. (1.70, 15.27);
\draw [yellow, fill=yellow, fill opacity=0.2, draw opacity=0.01] (8.36, 15.27) .. controls (8.36, 15.27) and (8.36, 12.37) .. (8.36, 12.37) .. controls (7.99, 12.77) and (7.99, 12.77) .. (7.61, 12.37) .. controls (7.61, 12.37) and (7.61, 15.27) .. (7.61, 15.27);
\draw [opacity=0.2] (2.04, 12.37) circle (0.35cm);
\draw [opacity=0.2] (8.03, 12.37) circle (0.35cm);
\begin{scope}[color=linkcolor1]
\draw (2.04, 15.27) .. controls (2.04, 15.27) and (2.04, 18) .. (2.04, 18);
\draw (2.04, 15.27) .. controls (2.04, 13.79) and (2.04, 12.30) .. (2.04, 10.82);
\draw (2.04, 9.93) .. controls (2.04, 9.93) and (2.04, 6) .. (2.04, 6);
\draw (8.03, 9.93) .. controls (8.03, 9.93) and (8.03, 6) .. (8.03, 6);
\draw (8.03, 10.92) .. controls (8.03, 12.37) and (8.04, 13.82) .. (8.05, 18);
\end{scope}
\draw [yellow, fill=yellow, fill opacity=0.2, draw opacity=0.01] (1.70, 15.27) .. controls (0.67, 15.27) and (0.34, 14.00) .. (0.32, 12.81) .. controls (0.31, 11.57) and (0.93, 10.36) .. (2.04, 10.37) .. controls (4.03, 10.41) and (6.03, 10.44) .. (8.02, 10.47) .. controls (9.10, 10.49) and (9.64, 11.68) .. (9.63, 12.88) .. controls (9.61, 14.03) and (9.35, 15.27) .. (8.36, 15.27) .. controls (8.36, 15.27) and (8.36, 12.37) .. (8.36, 12.37) .. controls (7.99, 11.87) and (7.99, 11.87) .. (7.61, 12.37) .. controls (7.61, 12.37) and (7.61, 15.27) .. (7.61, 15.27) .. controls (7.61, 15.27) and (2.49, 15.27) .. (2.49, 15.27) .. controls (2.49, 15.27) and (2.49, 12.37) .. (2.49, 12.37) .. controls (2.08, 11.77) and (2.08, 11.77) .. (1.70, 12.37);
\begin{scope}[color=linkcolor0]
\draw (1.70, 15.27) .. controls (0.67, 15.27) and (0.34, 14.00) .. (0.32, 12.81) .. controls (0.31, 11.57) and (0.93, 10.36) .. (2.04, 10.37);
\draw (2.04, 10.37) .. controls (4.03, 10.41) and (6.03, 10.44) .. (8.02, 10.47);
\draw (8.02, 10.47) .. controls (9.10, 10.49) and (9.64, 11.68) .. (9.63, 12.88) .. controls (9.61, 14.03) and (9.35, 15.27) .. (8.36, 15.27);
\draw (7.61, 15.27) .. controls (5.90, 15.27) and (4.19, 15.27) .. (2.49, 15.27);
\end{scope}
\end{tikzpicture}
\caption{An essential $2$-punctured disk from an unknotted component.}\label{fig:2Disk}
\end{figure}
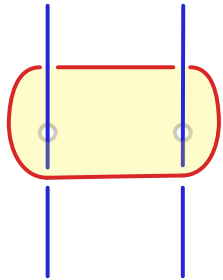

\medskip

The belted sum of noncommensurable arithmetic links will be a non-arithmetic link built out of two dissimilar hyperbolic building blocks. To prove Theorem \ref{thm:LinkMain}, we use an arithmetic link containing two trivial summands as in Figure \ref{fig:2Disk}. See Figure \ref{fig:ArithLink} for some arithmetic links with this property. We chose these links to be mutually incommensurable. The first is the Whitehead link, which is commensurable with $\PSL_2(\Z[i])$. The second is the three chain link, which is commensurable with $\PSL_2\left(\Z[\frac{1+\sqrt{-7}}{2}]\right)$. The third is a five component link commensurable with $\PSL_2\left(\Z[\frac{1+\sqrt{-15}}{2}]\right)$. See \cite[\S 9.2]{MaclachlanReid}.

\begin{figure}
\definecolor{linkcolor0}{rgb}{0.85, 0.15, 0.15}
\definecolor{linkcolor1}{rgb}{0.15, 0.15, 0.85}
\definecolor{linkcolor2}{rgb}{0.15, 0.85, 0.15}
\definecolor{linkcolor3}{rgb}{0.15, 0.85, 0.85}
\definecolor{linkcolor4}{rgb}{0.85, 0.15, 0.85}
\begin{tikzpicture}[line width=1, line cap=round, line join=round, scale=0.25]
  \begin{scope}[color=linkcolor0]
    \draw (0.88, 6.93) .. controls (0.38, 6.94) and (0.14, 6.37) ..
          (0.14, 5.80) .. controls (0.14, 5.20) and (0.52, 4.66) .. (1.08, 4.65);
    \draw (1.08, 4.65) .. controls (1.76, 4.64) and (2.45, 4.63) .. (3.13, 4.62);
    \draw (3.13, 4.62) .. controls (3.70, 4.62) and (4.12, 5.15) ..
          (4.11, 5.75) .. controls (4.11, 6.33) and (3.86, 6.90) .. (3.35, 6.91);
    \draw (2.95, 6.91) .. controls (2.39, 6.92) and (1.83, 6.92) .. (1.28, 6.93);
  \end{scope}
  \begin{scope}[color=linkcolor1]
    \draw (7.32, 7.43) .. controls (6.34, 7.46) and (6.37, 5.90) ..
          (6.41, 4.58) .. controls (6.44, 3.17) and (5.99, 1.71) ..
          (4.79, 1.75) .. controls (3.63, 1.78) and (3.11, 3.11) .. (3.13, 4.41);
    \draw (3.13, 4.83) .. controls (3.14, 5.53) and (3.15, 6.22) .. (3.16, 6.91);
    \draw (3.16, 6.91) .. controls (3.17, 8.41) and (3.92, 9.87) ..
          (5.27, 9.86) .. controls (6.44, 9.84) and (7.33, 8.83) .. (7.33, 7.64);
    \draw (7.32, 7.22) .. controls (7.32, 6.47) and (8.02, 5.93) .. (8.82, 5.93);
    \draw (8.82, 5.93) .. controls (9.76, 5.92) and (9.77, 7.29) ..
          (9.79, 8.46) .. controls (9.80, 10.33) and (7.57, 11.05) ..
          (5.43, 11.10) .. controls (3.05, 11.16) and (1.06, 9.29) .. (1.07, 6.93);
    \draw (1.07, 6.93) .. controls (1.07, 6.24) and (1.07, 5.55) .. (1.08, 4.86);
    \draw (1.08, 4.44) .. controls (1.09, 2.21) and (2.77, 0.34) ..
          (4.97, 0.24) .. controls (7.46, 0.13) and (8.84, 2.89) .. (8.82, 5.72);
    \draw (8.81, 6.14) .. controls (8.81, 6.88) and (8.11, 7.41) .. (7.32, 7.43);
  \end{scope}
\end{tikzpicture} \quad
\begin{tikzpicture}[line width=1, line cap=round, line join=round, scale=0.35]
  \begin{scope}[color=linkcolor0]
    \draw (0.80, 5.50) .. controls (0.27, 5.51) and (0.14, 4.84) ..
          (0.13, 4.22) .. controls (0.12, 3.58) and (0.38, 2.92) .. (0.95, 2.91);
    \draw (0.95, 2.91) .. controls (1.30, 2.90) and (1.64, 2.89) .. (1.99, 2.89);
    \draw (1.99, 2.89) .. controls (2.74, 2.87) and (3.39, 3.42) ..
          (3.40, 4.14) .. controls (3.41, 4.84) and (2.88, 5.44) .. (2.19, 5.46);
    \draw (1.79, 5.47) .. controls (1.58, 5.48) and (1.37, 5.48) .. (1.16, 5.49);
  \end{scope}
  \begin{scope}[color=linkcolor1]
    \draw (9.26, 5.20) .. controls (9.26, 7.00) and (7.17, 7.78) ..
          (5.12, 7.80) .. controls (3.10, 7.82) and (0.97, 7.24) .. (0.96, 5.50);
    \draw (0.96, 5.50) .. controls (0.96, 4.70) and (0.95, 3.90) .. (0.95, 3.11);
    \draw (0.94, 2.71) .. controls (0.93, 0.92) and (3.05, 0.21) ..
          (5.10, 0.18) .. controls (7.29, 0.14) and (9.26, 1.60) .. (9.26, 3.68);
    \draw (9.26, 3.68) .. controls (9.26, 4.05) and (9.26, 4.42) .. (9.26, 4.80);
  \end{scope}
  \begin{scope}[color=linkcolor2]
    \draw (8.62, 4.21) .. controls (7.95, 4.76) and (7.21, 5.23) ..
          (6.41, 5.56) .. controls (5.52, 5.93) and (4.58, 6.14) ..
          (3.62, 6.27) .. controls (2.82, 6.38) and (1.99, 6.14) .. (1.99, 5.47);
    \draw (1.99, 5.47) .. controls (1.99, 4.67) and (1.99, 3.88) .. (1.99, 3.09);
    \draw (1.99, 2.69) .. controls (1.99, 2.00) and (2.78, 1.66) ..
          (3.57, 1.74) .. controls (4.53, 1.84) and (5.47, 2.05) ..
          (6.36, 2.43) .. controls (7.19, 2.78) and (7.91, 3.36) .. (8.49, 4.06);
    \draw (8.74, 4.37) .. controls (8.92, 4.58) and (9.09, 4.79) .. (9.26, 5.00);
    \draw (9.26, 5.00) .. controls (9.44, 5.21) and (9.81, 4.95) ..
          (9.80, 4.45) .. controls (9.80, 3.89) and (9.66, 3.34) .. (9.37, 3.59);
    \draw (9.13, 3.78) .. controls (8.96, 3.93) and (8.79, 4.07) .. (8.62, 4.21);
  \end{scope}
\end{tikzpicture} \quad
\begin{tikzpicture}[line width=0.75, line cap=round, line join=round,scale=0.5]
  \begin{scope}[color=linkcolor0]
    \draw (6.32, 5.66) .. controls (6.32, 5.94) and (5.98, 6.03) ..
          (5.67, 6.03) .. controls (5.34, 6.04) and (5.01, 5.88) .. (5.02, 5.58);
    \draw (5.02, 5.58) .. controls (5.02, 5.35) and (5.02, 5.11) .. (5.02, 4.87);
    \draw (5.02, 4.87) .. controls (5.03, 4.55) and (5.34, 4.33) ..
          (5.69, 4.32) .. controls (6.02, 4.32) and (6.35, 4.47) .. (6.35, 4.76);
    \draw (6.34, 5.00) .. controls (6.34, 5.15) and (6.33, 5.30) .. (6.33, 5.45);
  \end{scope}
  \begin{scope}[color=linkcolor1]
    \draw (6.33, 5.57) .. controls (5.93, 5.58) and (5.53, 5.58) .. (5.14, 5.58);
    \draw (4.90, 5.59) .. controls (2.73, 5.60) and (0.65, 4.43) .. (0.67, 2.46);
    \draw (0.67, 2.46) .. controls (0.67, 1.55) and (1.98, 1.54) ..
          (3.11, 1.52) .. controls (4.08, 1.51) and (5.54, 1.49) .. (5.94, 1.75);
    \draw (6.13, 1.88) .. controls (6.44, 2.08) and (6.83, 2.12) .. (7.20, 2.12);
    \draw (7.43, 2.11) .. controls (7.69, 2.11) and (7.94, 2.11) .. (8.19, 2.11);
    \draw (8.19, 2.11) .. controls (8.70, 2.11) and (9.30, 2.11) .. (9.30, 1.71);
    \draw (9.30, 1.71) .. controls (9.30, 0.12) and (6.80, 0.11) ..
          (4.69, 0.10) .. controls (2.69, 0.09) and (0.08, 0.08) ..
          (0.08, 1.01) .. controls (0.08, 1.52) and (0.20, 2.04) .. (0.58, 2.38);
    \draw (0.75, 2.53) .. controls (1.08, 2.82) and (1.53, 2.88) .. (1.97, 2.87);
    \draw (2.20, 2.87) .. controls (2.45, 2.87) and (2.69, 2.86) .. (2.93, 2.86);
    \draw (3.17, 2.86) .. controls (3.38, 2.86) and (3.59, 2.85) .. (3.81, 2.85);
    \draw (3.81, 2.85) .. controls (4.32, 2.85) and (4.84, 2.73) .. (5.25, 2.42);
    \draw (5.25, 2.42) .. controls (5.51, 2.22) and (5.78, 2.02) .. (6.04, 1.81);
    \draw (6.04, 1.81) .. controls (6.37, 1.56) and (6.78, 1.46) .. (7.20, 1.45);
    \draw (7.43, 1.45) .. controls (7.69, 1.44) and (7.95, 1.44) .. (8.20, 1.43);
    \draw (8.20, 1.43) .. controls (8.55, 1.43) and (8.92, 1.46) .. (9.21, 1.65);
    \draw (9.40, 1.77) .. controls (9.87, 2.07) and (9.87, 3.04) ..
          (9.87, 3.80) .. controls (9.87, 5.25) and (8.02, 5.55) .. (6.33, 5.57);
  \end{scope}
  \begin{scope}[color=linkcolor2]
    \draw (4.90, 4.87) .. controls (3.50, 4.87) and (2.12, 4.16) .. (2.09, 2.87);
    \draw (2.09, 2.87) .. controls (2.07, 2.45) and (2.49, 2.17) .. (2.94, 2.16);
    \draw (3.18, 2.16) .. controls (3.39, 2.16) and (3.60, 2.16) .. (3.81, 2.15);
    \draw (3.81, 2.15) .. controls (4.27, 2.15) and (4.74, 2.15) .. (5.15, 2.36);
    \draw (5.36, 2.47) .. controls (6.07, 2.84) and (6.89, 2.92) ..
          (7.70, 2.90) .. controls (8.43, 2.89) and (9.14, 3.23) ..
          (9.14, 3.88) .. controls (9.13, 4.88) and (7.62, 4.88) .. (6.34, 4.88);
    \draw (6.34, 4.88) .. controls (5.94, 4.88) and (5.54, 4.88) .. (5.14, 4.87);
  \end{scope}
  \begin{scope}[color=linkcolor3]
    \draw (3.81, 2.93) .. controls (3.80, 3.12) and (3.62, 3.24) ..
          (3.42, 3.24) .. controls (3.21, 3.25) and (3.04, 3.07) .. (3.05, 2.86);
    \draw (3.05, 2.86) .. controls (3.05, 2.63) and (3.05, 2.39) .. (3.06, 2.16);
    \draw (3.06, 2.16) .. controls (3.06, 1.94) and (3.23, 1.76) ..
          (3.44, 1.75) .. controls (3.64, 1.74) and (3.81, 1.88) .. (3.81, 2.07);
    \draw (3.81, 2.27) .. controls (3.81, 2.43) and (3.81, 2.58) .. (3.81, 2.73);
  \end{scope}
  \begin{scope}[color=linkcolor4]
    \draw (8.19, 2.18) .. controls (8.19, 2.38) and (7.97, 2.48) ..
          (7.75, 2.48) .. controls (7.52, 2.48) and (7.32, 2.33) .. (7.32, 2.11);
    \draw (7.32, 2.11) .. controls (7.32, 1.89) and (7.32, 1.67) .. (7.32, 1.45);
    \draw (7.32, 1.45) .. controls (7.32, 1.16) and (7.50, 0.90) ..
          (7.76, 0.89) .. controls (8.01, 0.89) and (8.21, 1.08) .. (8.20, 1.32);
    \draw (8.20, 1.55) .. controls (8.20, 1.70) and (8.20, 1.85) .. (8.19, 1.99);
  \end{scope}
\end{tikzpicture}
\caption{Three incommensurable arithmetic links.}\label{fig:ArithLink}
\end{figure}
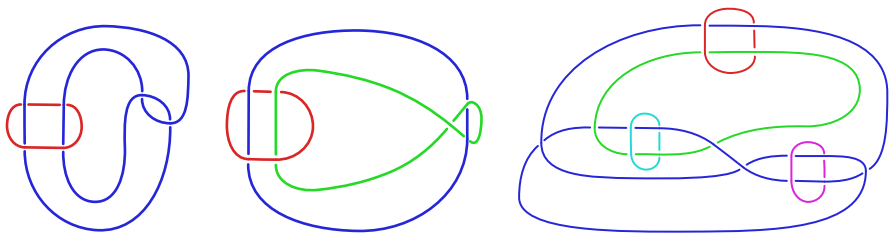

See Figure \ref{fig:FirstBelt} for the belted sum of the Whitehead link and the three chain link. This is the link $7_6^2$ from Rolfsen's tables \cite{Rolfsen}. This is a link complement built from dissimilar hyperbolic building blocks, hence it contains only finitely many immersed totally geodesic submanifolds. It is alternating, so it has no closed immersed totally geodesic submanifolds by \cite[Thm.\ 1]{MenascoReid}. We note that it has invariant trace field $\Q(i, \sqrt{-7})$ by a theorem of Neumann--Reid (see \cite[\S 5.6]{MaclachlanReid}). However, to prove Theorem \ref{thm:LinkMain}, we use the five component link in Figure \ref{fig:ArithLink}.

\begin{figure}
\definecolor{linkcolor0}{rgb}{0.85, 0.15, 0.15}
\definecolor{linkcolor1}{rgb}{0.15, 0.15, 0.85}
\begin{tikzpicture}[line width=1, line cap=round, line join=round, scale=0.5]
  \begin{scope}[color=linkcolor0]
    \draw (5.93, 6.06) .. controls (5.92, 6.49) and (5.50, 6.79) ..
          (5.04, 6.80) .. controls (4.60, 6.82) and (4.16, 6.61) .. (4.15, 6.21);
    \draw (4.15, 5.90) .. controls (4.15, 5.74) and (4.14, 5.57) .. (4.14, 5.41);
    \draw (4.14, 5.08) .. controls (4.13, 4.58) and (4.53, 4.17) ..
          (5.03, 4.15) .. controls (5.57, 4.13) and (5.94, 4.67) .. (5.93, 5.26);
    \draw (5.93, 5.26) .. controls (5.93, 5.52) and (5.93, 5.79) .. (5.93, 6.06);
  \end{scope}
  \begin{scope}[color=linkcolor1]
    \draw (7.75, 4.32) .. controls (7.74, 5.01) and (6.91, 5.26) .. (6.11, 5.26);
    \draw (5.76, 5.26) .. controls (5.22, 5.26) and (4.68, 5.25) .. (4.14, 5.25);
    \draw (4.14, 5.25) .. controls (3.08, 5.25) and (1.96, 4.96) .. (1.99, 4.05);
    \draw (1.99, 3.70) .. controls (2.00, 3.36) and (2.01, 3.03) .. (2.02, 2.69);
    \draw (2.02, 2.69) .. controls (2.05, 1.66) and (3.35, 1.48) ..
          (4.54, 1.48) .. controls (5.74, 1.49) and (7.03, 1.77) ..
          (7.05, 2.82) .. controls (7.05, 3.45) and (7.20, 4.14) .. (7.75, 4.14);
    \draw (7.75, 4.14) .. controls (8.44, 4.14) and (8.99, 3.56) .. (9.00, 2.85);
    \draw (9.00, 2.50) .. controls (9.03, 0.68) and (6.76, 0.14) ..
          (4.63, 0.14) .. controls (2.65, 0.13) and (0.21, 0.12) ..
          (0.23, 1.33) .. controls (0.24, 2.04) and (0.52, 2.74) .. (1.08, 3.17);
    \draw (1.36, 3.39) .. controls (1.57, 3.55) and (1.78, 3.71) .. (1.99, 3.88);
    \draw (1.99, 3.88) .. controls (2.38, 4.18) and (2.96, 3.86) ..
          (2.96, 3.31) .. controls (2.96, 2.72) and (2.49, 2.34) .. (2.16, 2.58);
    \draw (1.88, 2.79) .. controls (1.66, 2.95) and (1.44, 3.12) .. (1.22, 3.28);
    \draw (1.22, 3.28) .. controls (0.63, 3.72) and (0.15, 4.33) ..
          (0.14, 5.06) .. controls (0.12, 6.06) and (2.36, 6.06) .. (4.15, 6.06);
    \draw (4.15, 6.06) .. controls (4.68, 6.06) and (5.22, 6.06) .. (5.75, 6.06);
    \draw (6.10, 6.06) .. controls (7.83, 6.06) and (9.76, 5.86) ..
          (9.80, 4.36) .. controls (9.83, 3.55) and (9.69, 2.67) .. (9.00, 2.68);
    \draw (9.00, 2.68) .. controls (8.31, 2.69) and (7.76, 3.27) .. (7.75, 3.97);
  \end{scope}
\end{tikzpicture}
\caption{The belted sum of the Whitehead link and the three chain link.}\label{fig:FirstBelt}
\end{figure}
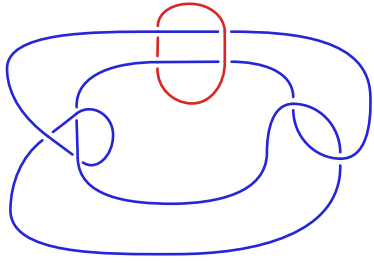

\begin{proof}[Proof of Theorem \ref{thm:LinkMain}]
Let $L_1$ be the Whitehead link and $L_2$ the five component link in Figure \ref{fig:ArithLink}. Then $L_2$ contains a pair of totally geodesic $3$-punctured spheres $S_2^+$ and $S_2^-$ associated with any two of the three unknotted components. For $r \ge 3$, let $L_r$ be the two component link obtained by performing $1/r$-Dehn surgery on the third unknotted component of $L_2$. Note that $S_2^+$ and $S_2^-$ still determine totally geodesic thrice-punctured spheres in $S^3 \ssm L_r$; choose one and call it $S_r$.

Let $N_1$ and be $S^3 \ssm L_1$ cut open along the totally geodesic $3$-punctured sphere $S_1$ associated with its unknotted component. Let $N_2$ be a copy of $S^3 \ssm L_2$, cut open along $S_2^+$ and $S_2^-$. Finally, let $N_r$ be the result of cutting $S^3 \ssm L_r$ open along $S_r$, $r \ge 3$. We can then perform an iterated belted sum, starting with $N_1$ belted to $N_2$ by gluing $S_1$ to $S_2^+$, then $N_2$ belted to $N_r$ by gluing $S_2^-$ to $S_r$.
\[
N_1 \overset{S_1 \leftrightarrow S_2^+}{\longleftrightarrow} N_2 \overset{S_2^- \leftrightarrow S_r}{\longleftrightarrow} N_r
\]
This determines a link $\calL_r$ in $S^3$ that is built from building blocks, and there are two adjacent blocks that are arithmetic and dissimilar, so $S^3 \ssm \calL_r$ contains only finitely many immersed totally geodesic submanifolds.

To complete the proof of the theorem, we must check that the links $\calL_r$ determine infinitely many distinct commensurability classes of hyperbolic $3$-manifolds. Let $k_r$ be the trace field of $S^3 \ssm \calL_r$. It suffices to show that the degree $[k_r : \Q]$ goes to infinity as $r$ goes to infinity. Let $F_r = \Q(\alpha_r)$ be the trace field of $S^3 \ssm L_r$. By a theorem of Neumann--Reid \cite[\S 5.6]{MaclachlanReid}, we have
\[
k_r = \Q(i, \sqrt{-7}, \alpha_r),
\]
so it now suffices to check that $[F_r : \Q]$ gets arbitrarily large. However, the links $L_r$ are all Dehn surgery on a fixed link (namely, $L_2$), so $S^3 \ssm L_r$ has uniformly bounded volume. Since there are only finitely many hyperbolic $3$-manifolds of bounded volume and trace field degree \cite{Jeon}, the result follows.
\end{proof}

\subsection{Hyperbolic Coxeter lattices  \label{subsec:Coxeter}}

A Coxeter polyhedron $P\subset \mathbb H^n$ is a finite-volume polyhedron with totally geodesic faces having the property that adjacent faces intersect at angles which are integral submultiples of $\pi$. Such finite-volume polyhedra can only exist in dimensions $n\leq 995$ by Prokhorov \cite{Prokhorov}, and compact ones can only exist in dimensions $n\leq 29$ by Vinberg \cite{Vinberg}. Associated with such a polyhedron $P$, one can form the lattice $\Gamma_P < \SO(n,1)$ generated by reflections in the hyperplanes containing the faces of the polyhedron. We call these \emph{Coxeter lattices}. Vinberg gave a simple criterion for whether such a lattice is arithmetic, and gave examples in dimensions $n=3,4,5$ of non-arithmetic Coxeter lattices that pre-dated the Gromov--Piatetski-Shapiro constructions. This was extended by Ruzmanov \cite{Ruzmanov} who produced non-arithmetic Coxeter lattices in dimension $6\leq n \leq 10$.

In a recent paper, Vinberg \cite{Vinberg2} points out that the
Gromov--Piatetski-Shapiro construction can also be applied to Coxeter polyhedra in a particularly simple manner. Given a pair
of polyhedra $P_1, P_2$, let us assume that they contain faces $F_i\subset P_i$ which are isometric to each other and have the
property that every other face which intersects $F_i$ does so orthogonally. In that case, we can form a new Coxeter polyhedron $P$
by gluing the $P_i$ together along the $F_i$. If the original polyhedra $P_i$ are not commensurable, then the resulting $P$ is non-arithmetic.
Vinberg used this procedure to create non-arithmetic Coxeter lattices in dimensions $n=11, 12, 14, 18$ (as well as new examples in some
lower dimensions).

Let us call a Coxeter polyhedron $P$ {\it splittable} if one can find a totally geodesic hyperplane $H\subset \mathbb H^n$ with the property
that every face of $P$ which intersects $H$ either is contained in $H$ or intersects $H$ orthogonally. Observe that the polyhedra $P_i$ in
Vinberg's gluing construction are splittable, as is the resulting polyhedron $P$. Splittable polyhedra also play a key role in
work of Allcock \cite{allcock}.
We call a polyhedron {\it unsplittable} if it is not splittable.

\begin{lemma}\label{lem-unsplittable-polyhedron}
Let $P$ be an unsplittable Coxeter polyhedron. Then the associated Coxeter lattice $\Gamma_P$ is not commensurable with any
Gromov--Piatetski-Shapiro type lattice.
\end{lemma}

\begin{proof}
By way of contradiction, assume that $\Gamma _P$ is commensurable with a Gromov--Piatetski-Shapiro type lattice $\Lambda$. Then
there are isomorphic finite index torsion-free subgroups $\Gamma < \Gamma_P$ and $\Gamma < \Lambda$. The associated hyperbolic
manifold $M = \mathbb H^n/\Gamma$ is then itself a Gromov--Piatetski-Shapiro type hyperbolic manifold,
but is also tessellated by finitely many copies of the hyperbolic polyhedron $P$. Let $\Sigma \subset M$ be a totally geodesic cutting
hypersurface for the Gromov--Piatetski-Shapiro decomposition of $M$.

In the tessellation
of $M$, consider any of the copies of $P$ that intersect $\Sigma$. Since $P$ is unsplittable and $\Sigma$ is codimension one, there
exists a face of this polyhedron that intersects $\Sigma$ at an angle $0<\theta<\pi/2$. But this face extends to a
closed totally geodesic immersed hypersurface, which contradicts angle rigidity.
\end{proof}

\begin{figure}
\begin{center}
\begin{tikzpicture}

\draw (1,0) -- (2,0);
\draw (0,.5) -- (1,0);
\draw (2,0) -- (3, .5);
\draw (1,.96) -- (2,.96);
\draw (1,1.04) -- (2,1.04);
\draw (0,.5) -- (1,1);
\draw (2,1) -- (3, .5);

\draw[fill=white] (0, .5) circle [radius=.1];
\draw[fill=white] (1, 0) circle [radius=.1];
\draw[fill=white] (2, 0) circle [radius=.1];
\draw[fill=white] (3, .5) circle [radius=.1];
\draw[fill=white] (1, 1) circle [radius=.1];
\draw[fill=white](2, 1) circle [radius=.1];

\end{tikzpicture}
\caption{The Coxeter diagram of a $5$-simplex in $\Hy^5$.}\label{fig:Coxeter5D}
\end{center}
\end{figure}
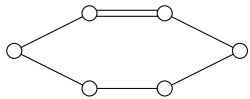

We now produce examples in dimension $3$ and $5$ to which Lemma \ref{lem-unsplittable-polyhedron} applies. First we briefly recall Vinberg's notion of a \emph{quasi-arithmetic} lattice. Let $\calG$ be a $\Q$-algebraic group such that $\calG(\Z)$ defines an arithmetic lattice in $\SO(n,1)$. A lattice $\Ga$ in $\SO(n,1)$ is quasi-arithmetic if it is contained in $\calG(\Q)$. Vinberg showed that there are quasi-arithmetic lattices that are not arithmetic (i.e., are not commensurable with $\calG(\Z)$) \cite{Vinberg3}. The inbreeding examples of Agol and Belolipetsky--Thomson are also quasi-arithmetic, but the Gromov--Piatetski-Shapiro type examples described in \S \ref{ssec:GPS} are not (see \cite{Thomson}).

\medskip

\noindent The $5$-dimensional example in the following will prove Theorem \ref{thm:MainCoxeter}, i.e., that there are non-arithmetic lattices in $\SO(5,1)$ that are not commensurable with a lattice constructed by the methods of Gromov--Piatetski-Shapiro or Agol.

\begin{ex}
It is easy to see that if the Coxeter polyhedron $P \subset \mathbb H^n$ is combinatorially an $n$-simplex, then $P$ is unsplittable.
There are 72 finite volume Coxeter polyhedra with this combinatorial type. Of these, a few examples give rise to non-arithmetic lattices:
one in dimension $5$ and seven in dimension $3$ (see \cite[p.\ 128]{JKRT}).
The $5$-dimensional example arises from a non-compact finite-volume $5$-simplex in $\mathbb H^5$ with
Coxeter diagram given in Figure \ref{fig:Coxeter5D}. For the $3$-dimensional examples, six of the seven examples are non-compact
and finite volume, while one is compact. Their Coxeter diagrams are listed out in Figure \ref{fig:Coxeter3D}, with the upper left-most
diagram corresponding to the compact $3$-simplex polyhedron.
For all these examples, our Lemma \ref{lem-unsplittable-polyhedron} applies and the associated lattices are not commensurable with any Gromov--Piatetski-Shapiro
type lattices. Moreover, for the non-compact polyhedra, the associated lattices are not quasi-arithmetic (see \cite[p.\ 442, Remark 3]{Vinberg3}). This proves Theorem \ref{thm:MainCoxeter}.
\end{ex}

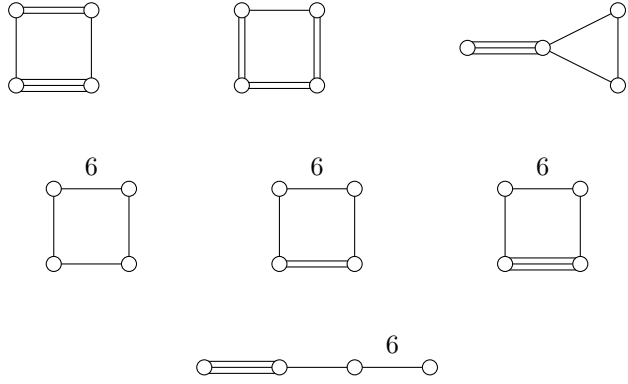
\begin{figure}
\begin{center}
\begin{tikzpicture}

\draw (-3,0) -- (-2,0);
\draw (-3, .08) -- (-2, .08);
\draw (-3, -.08) -- (-2, -.08);
\draw (-3,0) -- (-3,1);
\draw (-2,0) -- (-2,1);

\draw (-3,.96) -- (-2,.96);
\draw (-3,1.04) -- (-2,1.04);

\draw[fill=white] (-3, 0) circle [radius=.1];
\draw[fill=white] (-2, 0) circle [radius=.1];
\draw[fill=white] (-3, 1) circle [radius=.1];
\draw[fill=white](-2, 1) circle [radius=.1];

\draw (0, .04) -- (1, .04);
\draw (0, -.04) -- (1, -.04);
\draw (-.04,0) -- (-.04,1);
\draw (.04,0) -- (.04,1);
\draw (.96,0) -- (.96,1);
\draw (1.04,0) -- (1.04,1);

\draw (0,1) -- (1,1);

\draw[fill=white] (0, 0) circle [radius=.1];
\draw[fill=white] (1, 0) circle [radius=.1];
\draw[fill=white] (0, 1) circle [radius=.1];
\draw[fill=white](1, 1) circle [radius=.1];

\draw (3, 0.5) -- (4,0.5);
\draw (3, .58) -- (4, .58);
\draw (3, .42) -- (4, .42);
\draw (4, 0.5) -- (5, 1);
\draw (4, 0.5) -- (5, 0);
\draw (5,0) -- (5, 1);

\draw[fill=white] (3, 0.5) circle [radius=.1];
\draw[fill=white] (4, 0.5) circle [radius=.1];
\draw[fill=white] (5, 0) circle [radius=.1];
\draw[fill=white] (5, 1) circle [radius=.1];

\end{tikzpicture}

\vskip 20pt

\begin{tikzpicture}

\draw (-3, 0) -- (-2, 0);
\draw (-3,0) -- (-3,1);
\draw (-2,0) -- (-2,1);
\draw (-3,1) -- (-2,1);

\node at (-2.5, 1.3) {6};

\draw[fill=white] (-3, 0) circle [radius=.1];
\draw[fill=white] (-2, 0) circle [radius=.1];
\draw[fill=white] (-3, 1) circle [radius=.1];
\draw[fill=white](-2, 1) circle [radius=.1];

\draw (0, .04) -- (1, .04);
\draw (0, -.04) -- (1, -.04);
\draw (0,0) -- (0,1);
\draw (1,0) -- (1,1);
\draw (0,1) -- (1,1);

\node at (0.5, 1.3) {6};

\draw[fill=white] (0, 0) circle [radius=.1];
\draw[fill=white] (1, 0) circle [radius=.1];
\draw[fill=white] (0, 1) circle [radius=.1];
\draw[fill=white](1, 1) circle [radius=.1];

\draw (3,0) -- (4,0);
\draw (3, .08) -- (4, .08);
\draw (3, -.08) -- (4, -.08);
\draw (3,0) -- (3,1);
\draw (4,0) -- (4,1);
\draw (3,1) -- (4,1);

\node at (3.5, 1.3) {6};

\draw[fill=white] (3, 0) circle [radius=.1];
\draw[fill=white] (4, 0) circle [radius=.1];
\draw[fill=white] (3, 1) circle [radius=.1];
\draw[fill=white](4, 1) circle [radius=.1];

\end{tikzpicture}

\vskip 20pt

\begin{tikzpicture}

\draw (0, 0) -- (1, 0);
\draw (0, 0.08) -- (1, 0.08);
\draw (0, -0.08) -- (1, -0.08);
\draw (1,0) -- (2,0);
\draw (2,0) -- (3,0);

\node at (2.5, 0.3) {6};

\draw[fill=white] (0, 0) circle [radius=.1];
\draw[fill=white] (1, 0) circle [radius=.1];
\draw[fill=white] (2, 0) circle [radius=.1];
\draw[fill=white](3, 0) circle [radius=.1];

\end{tikzpicture}

\caption{Coxeter diagrams for $3$-simplices in $\Hy^3$.}\label{fig:Coxeter3D}
\end{center}
\end{figure}

Given a splittable Coxeter polyhedron $P\subset \mathbb H^n$, one can consider the hyperplane $H\subset \mathbb H^n$
giving rise to the splitting. Then the intersection $H\cap P$ yields a Coxeter polyhedron in $H \cong \mathbb H^{n-1}$ with
faces of $H\cap P$ corresponding bijectively with the faces of the original $P$ that were orthogonal to $H$. The number of
faces of $H\cap P$ is strictly smaller than the number of faces of $P$. Indeed, $H$ separates $\mathbb H^n$ into two half-spaces
$H^+$ and $H^-$, and in order for $P$ to have finite volume, there needs to be at least one face of $P$ on each side of $H$
satisfying $H^\pm \cap \mathrm{Int}(P) \neq \emptyset$. Such faces will be disjoint from $H$, giving a generator for $\Gamma_P$
which does not lie in $\Gamma_{P\cap H}$. Recall that a {\it special subgroup} of a Coxeter group is a subgroup generated
by a subset of the Coxeter generating set. This establishes the following:

\begin{lemma}\label{lem:SplitMe}
If a Coxeter polyhedron $P\in \mathbb H^n$ is splittable, then the associated Coxeter group
contains a proper special subgroup which is itself a Coxeter group for a Coxeter polyhedron $P'\subset \mathbb H^{n-1}$.
\end{lemma}

In practice, Lemma \ref{lem:SplitMe} can be used to show that certain Coxeter polyhedron are unsplittable. For example, Andreev's
theorem gives an algorithm for deciding whether a Coxeter group arises from a polyhedron in $\mathbb H^3$ (see
\cite[\S 6.10]{davis}). Thus if
one is given a Coxeter polyhedron $P\subset \mathbb H^4$, one can sometimes use Andreev's theorem to algorithmically verify
if the polyhedron is unsplittable.
Unfortunately all examples of $4$-dimensional Coxeter polyhedra currently known to the authors are either splittable, or yield
lattices that are quasi-arithmetic.

Closely related to our unsplittable polyhedra are the {\it essential} polyhedra, introduced by Felikson and Tumarkin \cite{FT}.
These are polyhedra which cannot be decomposed into smaller Coxeter polyhedra glued along faces. In their paper they give a
commensurability classification of all the known (at the time) essential polyhedra. While the two notions of unsplittable
and essential are related, neither implies the other.

\medskip

We close by noting that it would be interesting to use results of this kind to produce infinitely many commensurability classes of lattices in $\SO(n,1)$ for small $n$ that cannot arise from  the constructions of Gromov--Piatetski-Shapiro or Agol. While this seems possible in small dimensions, one still wonders if the ``maybe for large $n$'' in Question \ref{qtn:GPS} is related to the non-existence of Coxeter polytopes and related geometric objects in sufficiently high dimension.

\bibliographystyle{abbrv}
\bibliography{Biblio}

\begin{thebibliography}{10}

\bibitem{SevenSamuraiShort}
M.~Abert, N.~Bergeron, I.~Biringer, T.~Gelander, N.~Nikolov, J.~Raimbault, and
  I.~Samet.
\newblock On the growth of {$L^2$}-invariants of locally symmetric spaces,
  {II}: exotic invariant random subgroups in rank one.
\newblock arXiv:1612.09510.

\bibitem{Adams}
C.~C. Adams.
\newblock Thrice-punctured spheres in hyperbolic {$3$}-manifolds.
\newblock {\em Trans. Amer. Math. Soc.}, 287(2):645--656, 1985.

\bibitem{Agol}
I.~Agol.
\newblock Systoles of hyperbolic {$4$}-manifolds.
\newblock arXiv: 0612290.

\bibitem{allcock}
D.~Allcock.
\newblock Infinitely many hyperbolic {C}oxeter groups through dimension 19.
\newblock {\em Geom. Topol.}, 10:737--758, 2006.

\bibitem{BelThom}
M.~V. Belolipetsky and S.~A. Thomson.
\newblock Systoles of hyperbolic manifolds.
\newblock {\em Algebr. Geom. Topol.}, 11(3):1455--1469, 2011.

\bibitem{BenoistOh}
Y.~Benoist and H.~Oh.
\newblock Geodesic planes in geometrically finite acylindrical {$3$}-manifolds.
\newblock {\em Preprint}, 2018.

\bibitem{BorelCrelle}
A.~Borel.
\newblock Density and maximality of arithmetic subgroups.
\newblock {\em J. Reine Angew. Math.}, 224:78--89, 1966.

\bibitem{Calegari}
D.~Calegari.
\newblock Real places and torus bundles.
\newblock {\em Geom. Dedicata}, 118:209--227, 2006.

\bibitem{DaniMargulisOppenheim}
S.~G. Dani and G.~A. Margulis.
\newblock Values of quadratic forms at primitive integral points.
\newblock {\em Invent. Math.}, 98(2):405--424, 1989.

\bibitem{DaniMargulis}
S.~G. Dani and G.~A. Margulis.
\newblock Limit distributions of orbits of unipotent flows and values of
  quadratic forms.
\newblock In {\em I. {M}. {G}el'fand {S}eminar}, volume~16 of {\em Adv. Soviet
  Math.}, pages 91--137. Amer. Math. Soc., 1993.

\bibitem{davis}
M.~W. Davis.
\newblock {\em The geometry and topology of {C}oxeter groups}, volume~32 of
  {\em London Mathematical Society Monographs Series}.
\newblock Princeton University Press, 2008.

\bibitem{Curt}
K.~Delp, D.~Hoffoss, and J.~Manning.
\newblock Problems in {G}roups, {G}eometry, and {T}hree-{M}anifolds.
\newblock arXiv:1512.04620.

\bibitem{EinsiedlerWirth}
M.~Einsiedler and P.~Wirth.
\newblock Effective equidistribution of closed hyperbolic surfaces on
  congruence quotients of hyperbolic spaces.
\newblock {\em In preparation}, 2018.

\bibitem{FT}
A.~Felikson and P.~Tumarkin.
\newblock Essential hyperbolic {C}oxeter polytopes.
\newblock {\em Israel J. Math.}, 199(1):113--161, 2014.

\bibitem{GeLev}
T.~Gelander and A.~Levit.
\newblock Counting commensurability classes of hyperbolic manifolds.
\newblock {\em Geom. Funct. Anal.}, 24(5):1431--1447, 2014.

\bibitem{GPS}
M.~Gromov and I.~Piatetski-Shapiro.
\newblock Nonarithmetic groups in lobachevsky spaces.
\newblock {\em Inst. Hautes \'Etudes Sci. Publ. Math.}, (66):93--103, 1988.

\bibitem{Jeon}
B.~Jeon.
\newblock Hyperbolic three manifolds of bounded volume and trace field degree
  {II}.
\newblock arXiv:1409.2069.

\bibitem{JKRT}
N.~W. Johnson, R.~Kellerhals, J.~G. Ratcliffe, and S.~T. Tschantz.
\newblock Commensurability classes of hyperbolic {C}oxeter groups.
\newblock {\em Linear Algebra Appl.}, 345:119--147, 2002.

\bibitem{MRTG}
C.~Maclachlan and A.~W. Reid.
\newblock Commensurability classes of arithmetic {K}leinian groups and their
  {F}uchsian subgroups.
\newblock {\em Math. Proc. Cambridge Philos. Soc.}, 102(2):251--257, 1987.

\bibitem{MaclachlanReid}
C.~Maclachlan and A.~W. Reid.
\newblock {\em The arithmetic of hyperbolic 3-manifolds}, volume 219 of {\em
  Graduate Texts in Mathematics}.
\newblock Springer-Verlag, 2003.

\bibitem{Marden}
A.~Marden.
\newblock {\em Outer circles}.
\newblock Cambridge University Press, 2007.
\newblock An introduction to hyperbolic 3-manifolds.

\bibitem{MargulisOppenheim}
G.~A. Margulis.
\newblock Formes quadratriques ind\'efinies et flots unipotents sur les espaces
  homog\`enes.
\newblock {\em C. R. Acad. Sci. Paris S\'er. I Math.}, 304(10):249--253, 1987.

\bibitem{MargulisBook}
G.~A. Margulis.
\newblock {\em Discrete subgroups of semisimple Lie groups}, volume~17 of {\em
  Ergebnisse der Mathematik und ihrer Grenzgebiete (3)}.
\newblock Springer-Verlag, 1991.

\bibitem{MMO}
C.~T. McMullen, A.~Mohammadi, and H.~Oh.
\newblock Geodesic planes in hyperbolic 3-manifolds.
\newblock {\em Invent. Math.}, 209(2):425--461, 2017.

\bibitem{MenascoReid}
W.~Menasco and A.~W. Reid.
\newblock Totally geodesic surfaces in hyperbolic link complements.
\newblock In {\em Topology '90 ({C}olumbus, {OH}, 1990)}, volume~1 of {\em Ohio
  State Univ. Math. Res. Inst. Publ.}, pages 215--226. de Gruyter, 1992.

\bibitem{Meyer}
J.~S. Meyer.
\newblock Totally geodesic spectra of arithmetic hyperbolic spaces.
\newblock {\em Trans. Amer. Math. Soc.}, 369(11):7549--7588, 2017.

\bibitem{Millson}
J.~J. Millson.
\newblock On the first {B}etti number of a constant negatively curved manifold.
\newblock {\em Ann. of Math. (2)}, 104(2):235--247, 1976.

\bibitem{Prokhorov}
M.~N. Prokhorov.
\newblock Absence of discrete groups of reflections with a noncompact
  fundamental polyhedron of finite volume in a {L}obachevski\u\i \ space of
  high dimension.
\newblock {\em Izv. Akad. Nauk SSSR Ser. Mat.}, 50(2):413--424, 1986.

\bibitem{Raimbault}
J.~Raimbault.
\newblock A note on maximal lattice growth in {${\rm SO}(1,n)$}.
\newblock {\em Int. Math. Res. Not. IMRN}, (16):3722--3731, 2013.

\bibitem{Ratcliffe}
J.~G. Ratcliffe.
\newblock {\em Foundations of hyperbolic manifolds}, volume 149 of {\em
  Graduate Texts in Mathematics}.
\newblock Springer-Verlag, 1994.

\bibitem{Ratner}
M.~Ratner.
\newblock On {R}aghunathan's measure conjecture.
\newblock {\em Ann. of Math. (2)}, 134(3):545--607, 1991.

\bibitem{RatnerOrbitClosure}
M.~Ratner.
\newblock Raghunathan's topological conjecture and distributions of unipotent
  flows.
\newblock {\em Duke Math. J.}, 63(1):235--280, 1991.

\bibitem{ReidTG}
A.~W. Reid.
\newblock Totally geodesic surfaces in hyperbolic {$3$}-manifolds.
\newblock {\em Proc. Edinburgh Math. Soc. (2)}, 34(1):77--88, 1991.

\bibitem{Rolfsen}
D.~Rolfsen.
\newblock {\em Knots and links}, volume~7 of {\em Mathematics Lecture Series}.
\newblock Publish or Perish, 1990.

\bibitem{Ruzmanov}
O.~P. Ruzmanov.
\newblock Examples of nonarithmetic crystallographic {C}oxeter groups in
  {$n$}-dimensional {L}obachevski\u\i \ space when {$6\leq n\leq 10$}.
\newblock In {\em Problems in group theory and in homological algebra
  ({R}ussian)}, pages 138--142. Yaroslav. Gos. Univ., Yaroslavl', 1989.

\bibitem{Shah}
N.~A. Shah.
\newblock Closures of totally geodesic immersions in manifolds of constant
  negative curvature.
\newblock In {\em Group theory from a geometrical viewpoint ({T}rieste, 1990)},
  pages 718--732. World Sci. Publ., 1991.

\bibitem{Thomson}
S.~Thomson.
\newblock Quasi-arithmeticity of lattices in {$\rm{PO}(n,1)$}.
\newblock {\em Geom. Dedicata}, 180:85--94, 2016.

\bibitem{Vinberg3}
E.~B. Vinberg.
\newblock Discrete groups generated by reflections in {L}oba\v cevski\u\i \
  spaces.
\newblock {\em Mat. Sb. (N.S.)}, 72 (114):471--488; correction, ibid. 73 (115)
  (1967), 303, 1967.

\bibitem{Vinberg}
E.~B. Vinberg.
\newblock The nonexistence of crystallographic reflection groups in
  {L}obachevski\u\i \ spaces of large dimension.
\newblock {\em Funktsional. Anal. i Prilozhen.}, 15(2):67--68, 1981.

\bibitem{Vinberg2}
E.~B. Vinberg.
\newblock Non-arithmetic hyperbolic reflection groups in higher dimensions.
\newblock {\em Mosc. Math. J.}, 15(3):593--602, 606, 2015.

\end{thebibliography}

\end{document}